\documentclass[submission,copyright,creativecommons]{eptcs}
\providecommand{\event}{QPL 2017} 

\usepackage{amsfonts, amsthm, amssymb, amsmath, stmaryrd, etoolbox}
\usepackage{mathtools}
\usepackage{graphicx,caption,subcaption}
\usepackage{xcolor}
\usepackage{tikz}
\usetikzlibrary{matrix,arrows,shapes,decorations.markings,decorations.pathreplacing}

\renewcommand{\epsilon}{\varepsilon}

\newcommand{\cat}[1]{\mathbf{#1}}
\newcommand{\dblcat}[1]{\mathbb{#1}}
\renewcommand{\t}[1]{\text{#1}}

\newcommand{\from}{\colon}
\newcommand{\xto}[1]{\xrightarrow{#1}}

\newcommand{\bimonspcsp}[1]{\mathbf{MonicSp(Csp(#1))}}

\newcommand{\spcsp}[1]{\mathbf{Sp(Csp(#1))}}
\newcommand{\dblspcsp}[1]{\mathbb{S}\mathbf{p(Csp(#1))}}

\DeclareMathOperator{\colim}{colim}

\newtheorem{thm}{Theorem}[section]
\newtheorem{lem}[thm]{Lemma}

\theoremstyle{remark}

\theoremstyle{definition}
	\newtheorem{ex}[thm]{Example} 
	\newtheorem{defn}[thm]{Definition}
	
\tikzset{rewritenode/.style={shape=circle,fill=rewritecolor,scale=0.25,font=\Huge}}
\tikzset{RWopen/.style={shape=circle,draw=black,fill=white,scale=0.5,font=\Huge}}
\tikzset{RWclosed/.style={shape=circle,fill=black,scale=0.5,font=\Huge}}
\tikzset{CDnode/.style={shape=circle,fill=white,scale=.5}}
\tikzset{zxgreen/.style={shape=circle,draw,thick,fill=green}}
\tikzset{zxred/.style={shape=circle,draw,thick,fill=red}}
\tikzset{zxyellow/.style={shape=rectangle,draw,thick,fill=yellow}}
\tikzset{zxdiamond/.style={shape=diamond,fill=black,inner sep=2.75}}
\tikzset{zxopen/.style={shape=circle,draw,thick,inner sep=2pt}}
\tikzset{->-/.style={decoration={%
			markings,
			mark=at position .5 with {\arrow{>}}},postaction={decorate}}
}
\tikzset{->-pos/.style={decoration={%
			markings,
			mark=at position #1 with {\arrow{>}}},postaction={decorate}}
}
\tikzset{->-/.style={decoration={%
			markings,
			mark=at position .5 with {\arrow{>}}},postaction={decorate}}
}
\tikzset{->-pos/.style={decoration={%
			markings,
			mark=at position #1 with {\arrow{>}}},postaction={decorate}}
}

\title{Categorifying the ZX-calculus}
\author{Daniel Cicala
\institute{University of California, Riverside}
\institute{Department of Mathematics}
\email{cicala@math.ucr.edu}
}
\def\titlerunning{Categorifying the ZX-calculus}
\def\authorrunning{D.~Cicala}

\begin{document} 

\maketitle

\begin{abstract}
	We build a symmetric monoidal and 
	compact closed bicategory
	by combining spans and cospans
	inside a topos.
	This can be used as a framework
	in which to study open networks 
	and diagrammatic languages.
	We illustrate this framework with
	Coecke and Duncan's zx-calculus
	by constructing a bicategory with 
	the natural numbers for 0-cells,
	the zx-calculus diagrams for 1-cells, and
	rewrite rules for 2-cells. 
\end{abstract}

%
%

%
%

\section{Introduction}
\label{sec:Introduction}

Compositionality is increasingly
becoming recognized as 
a viable point of view
from which to study 
complex systems such as 
those found in physics
	\cite{AbramCoecke_CatSemanticQuantum}, 
computer science 
	\cite{SassoneSobocinski_PetriNets}, 
and biology 
	\cite{BaezFongPollard_CompMarkovProcesses}.  
The focus is on 
connecting together smaller, 
simpler systems.  
The word `compositionality' suggests 
the relevancy of category theory.
This is indeed the case.
In fact, this paper
fits into a 
larger project of
establishing suitable categorical frameworks
in which to study composable systems
\cite{BaezCoyaFong_Props,
BaezFong_CompPassLinNets,
BaezFongPollard_CompMarkovProcesses,
BaezPollard_CompFrameRxNets,
Cicala_SpansCospans,
CicalaCourser_BicatSpansCospan,
Pollard_OpenMarkov}.
Open diagrams and diagrammatic languages
are typical players in compositional approaches. 
In our context, the adjective `open' is 
an established term
\cite{Dixon_OpenGraphs,
Merry_BangGraphs,
Pollard_OpenMarkov}
referring to a structure equipped with 
chosen inputs and outputs. 
The primary advantage of
open diagrams 
is the ability to work
within a more intuitive syntax.
Such diagrams are typically 
constructed with graph or 
topological string-like objects.  
Occasionally, additional data 
is attached to nodes, edges, or strings 
as needed. 
For example, open Markov chains
	\cite{Pollard_OpenMarkov}
use graphs whose nodes are 
labeled with a \emph{population} 
and edges with the rate at which 
the population of the source node 
shifts to the target node.  

Another common feature shared 
between diagrammatic languages is 
the notion of equality.  
Formal languages are often
equipped with a collection of rewrite rules.
For us, a rewrite rule is an equivalence relation
on diagrams stating when we can
replace a diagram $D$ with diagram $D'$.
This is our equality. 

Currently, syntax for diagrammatic calculi 
are usually captured with 1-categories
by encoding diagrams as morphisms
and diagram connection by composition. 
As mentioned above,
rewrite rules provide a notion of
equality between diagrams. 
However, 1-categorical frameworks squash 
the information contained in a rewrite rule.
That is, there is no way to reconstruct 
a rewrite rule from an equality. 
To more fully capture a system, 
rewrite rules ought to have better
representation in our syntax.
We accomplish this is by 
including them as 2-cells in a bicategory.  

What should such a bicategory look like? 
The 0-cells should 
communicate whether a pair of diagrams
can be connected.
Taking 0-cells to be sets, 
then a 1-cell $D \from x \to y$ is 
a diagram whose set of inputs is $x$ 
and set of outputs is $y$. 
Hence, we can connect to $D$ any
diagram whose inputs are $y$ or 
outputs are $x$.  
The 2-cells $D \Rightarrow D'$ are 
rules that rewrite $D$ into $D'$.  

To better envision such a bicategory, 
consider a hypothetical system
modeled by a directed graph $D$. 
Suppose that we would like $D$ 
to have inputs $I$ and outputs $O$, 
each subsets of the $D$-nodes.  
We can build this information 
into a cospan $I \to D \gets O$
of graphs by taking $I$ and $O$ 
to be edgeless graphs.  
Our 1-cells are cospans like this.  
This means that the 
inputs and outputs are $0$-cells.
Rewrite rules are included 
through \emph{double pushout rewriting}
	\cite{Corradini_AlgApprGraphTrans}. 
This presents a rule rewriting $D$ to $D'$ as 
a span of graphs 
$D \gets K \to D'$, 
through some intermediary graph $K$. 
As 2-cells in our bicategory, rewrite rules
are isomorphism classes of spans of cospans. 
These are depicted in Figure \ref{fig:spans_of_cospans}.  
Kissinger
	\cite{Kissinger_Pictures}
also modeled rewriting using spans of cospans
under the term \emph{cospan rewrites}.	 

The author proved that 
this construction
actually gives a bicategory
	\cite{Cicala_SpansCospans}.  
In particular, 
starting with a topos $T$, there is a 
bicategory $\bimonspcsp{T}$ whose 
0-cells are the $T$-objects, 
1-cells are cospans in $T$, and
2-cells are isomorphism classes of 
monic legged spans of cospans in $T$.
A topos and monic span legs are required to
ensure that the interchange law holds. 
This is not overly restrictive,
because the spans used in 
double pushout rewriting are
often assumed to have 
monic legs \cite{Habel}.
Though not discussed in that paper, 
we can bypass both needs 
by taking coarser classes of 2-cells. 
Specifically, we consider the bicategory 
$\spcsp{C}$ 
where $\cat{C}$ is a category with 
finite limits and colimits. 
This differs from $\bimonspcsp{T}$ by 
taking 2-cells to be all spans of cospans
up to \emph{having the same domain and codomain}.

The reason for constructing
$\bimonspcsp{T}$ and $\spcsp{C}$ 
is to provide syntactic bicategories for 
diagrammatic languages. 
Which bicategory we use depends
on the nature of the diagrammatic
language of interest.
Regardless of which bicategory we use,
we typically start by letting
$\cat{T}$ or $\cat{C}$ be the 
topos $\cat{Graph}$ of directed graphs
or, perhaps, the topos consisting
of some other flavor of graphs. 
For now, we look at 
$\bimonspcsp{Graph}$ and $\spcsp{Graph}$
and consider, in each,
the sub-bicategory 
that is 1-full and 2-full 
on the edgeless graphs.  
The 1-cells this sub-bicategory are open graphs 
and the 2-cells are ways to rewrite 
one open graph to another.  
Because we are currently painting with broad strokes,
distinguishing between this sub-bicategory in
$\bimonspcsp{Graph}$ or $\spcsp{Graph}$
is inconsequential.
Hence, we commit the sin of
referring to this bicategory as $\cat{Rewrite}$
regardless of where it lives.

Suppose we have a 
diagrammatic language $L$
given by some presentation.  
We must find
a suitable way to identity  
the given generators and relations of $L$
with 1-cells and 2-cells, 
respectively, of $\cat{Rewrite}$.
These 1-cells and 2-cells, in turn, 
generate a sub-bicategory of $\cat{Rewrite}$ 
that gives a bicategorical syntax 
for $L$.  
It was shown by the author and Courser
	\cite{CicalaCourser_BicatSpansCospan} 
that $\bimonspcsp{T}$ is 
symmetric monoidal and compact closed
in the sense of Stay 
	\cite{Stay_CompactClosedBicats}.
We employ a similar argument
showing the same is true of $\spcsp{C}$.
Therefore $\cat{Rewrite}$ and 
all of sub-bicategories
we generate within $\cat{Rewrite}$ 
are symmetric monoidal and compact closed.

Due to using isomorphism classes for 2-cells
instead of any other equivalence classes,
it would seem that beginning with 
$\bimonspcsp{T}$ is the natural construction. 
Indeed, it is suitable for working with systems 
admitting a graphical syntax such as
the open Markov processes mentioned above. 
However, systems whose 
syntax has topological information, 
like string diagrams, 
introduce the challenge of 
conveying topological information 
with only graphs. 
To contend with this problem, 
we begin with $\spcsp{C}$ 
because it has a 2-cell 
not present in $\bimonspcsp{T}$.
This 2-cell rewrites an edge into a single node,
thus forces an analogy between an edge
and a string that behaves like an identity.
As we will see, this rewrite rule is given by 
a span with a non-monomorphic leg, 
leading us to use
$\spcsp{C}$ instead of $\bimonspcsp{T}$. 

The purpose of this paper is 
to illustrate our framework with
the zx-calculus.  
The backstory of the zx-calculus dates to 
Penrose's tensor networks
	\cite{Penrose_NegDimTensors} 
and, more recently, 
to the relationship between 
graphical languages and 
monoidal categories 
	\cite{JoyalStreet_GeomTensorCalc,Selinger_GraphicsMonCats}.  
Abramsky and Coecke 
capitalized on this relationship 
when inventing a categorical framework 
for quantum physics
	\cite{AbramCoecke_CatSemanticQuantum}.  
Soon after, 
Coecke and Duncan 
introduced a diagrammatic language 
in which to reason about 
\emph{complementary quantum observables} 
	\cite{CoeckeDuncan_QuantumObsInitialReport}. 
After a fruitful period of development
\cite{CoeckeEdwardsSpekkens_PhaseGrpsNonLocality,
CoeckePerdix_EnvironClassicChannels,
DuncanPerdix_GraphStatesEulerDecomp,
DuncanPerdrix_RewritingQuantumCompu,
EvansDuncanLangPanan_ClassMutualUnbias,
Pavlovic_QuanClassNondetermCompu}, 
a full presentation of 
the zx-calculus was published
	\cite{CoeckeDuncan_QuantumObsFullPaper}.
The completeness of 
the zx-calculus for 
stabilizer quantum mechanics 
was later shown by Backens 
	\cite{Backens_Completeness}.  

The zx-calculus begins with 
the five diagrams depicted 
in Figure \ref{fig:ZX_generators}.
\begin{figure}
	\fbox{%
		\begin{minipage}{0.975\textwidth}
			\centering
			\subcaptionbox{Wire}[2 cm]{%
				\centering
				\includegraphics{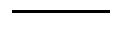}
			}
			\subcaptionbox{Green spider}[2.5 cm]{%
				\centering
				\includegraphics{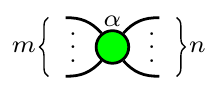}
			}
			\subcaptionbox{Red spider}[2.5 cm]{%
				\centering
				\includegraphics{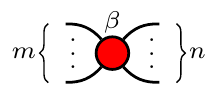}
			}
			\subcaptionbox{Hadamard}[2.5 cm]{%
				\centering
				\includegraphics[]{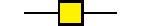}
			}
			\subcaptionbox{Diamond}[2 cm]{%
				\centering
				\includegraphics{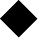}
			}
		\end{minipage}
	}
	\caption{Generators for the category $\mathbf{zx}$}
	\label{fig:ZX_generators}
\end{figure}
The dangling wires 
on the diagrams' left 
are \emph{inputs} and 
those on the right are \emph{outputs}. 
By connecting inputs to outputs, 
we can form larger diagrams.  
Formalizing this perspective, 
we let these diagrams 
generate the morphisms of 
a dagger compact category 
$\mathbf{zx}$ whose 
objects, the non-negative integers, 
count the inputs and outputs
of a diagram.  
Section 
	\ref{sec:ZxCalc} 
contains a presentation of $\mathbf{zx}$ 
along with a brief discussion on 
the origins of the generating morphisms 
(Figure \ref{fig:ZX_generators}) 
and relations (Figure \ref{fig:ZX_equations}). 
We also mention relevant software, 
Quantomatic 
	\cite{AbramCoecke_CatSemanticQuantum,
		KissingerZamd_Quantomatic} 
and Globular 
	\cite{BarKissingerVicary_Globular}.  

Our goal in this paper is
to generate a
symmetric monoidal and compact closed (SMCC) 
bicategory $\underline{\mathbf{zx}}$
that provides a syntax for the zx-calculus. 
Our first steps towards constructing 
$\underline{\mathbf{zx}}$
is in Section 
	\ref{sec:RewritingOpenGraphs} 
where we fit open graphs 
into an SMCC bicategory.
To this end, 
we slightly modify recent
work by Courser and the author
	\cite{Cicala_SpansCospans,
		CicalaCourser_BicatSpansCospan}
in order to produce an SMCC bicategory 
with graphs as 0-cells, 
cospans of graphs as 1-cells, 
and certain equivalence classes 
of spans of cospans of graphs as 2-cells
(see Figure \ref{fig:spans_of_cospans}). 
As discussed above, this has an
SMCC sub-bicategory $\mathbf{Rewrite}$
that provides an ambient space 
in which to generate systems 
modeled on open graphs.  

\begin{figure}
	\fbox{%
		\begin{minipage}{0.975\textwidth}
			\centering
			\subcaptionbox{A span of cospans}[5cm]{%
				\centering				
				\includegraphics{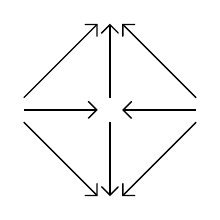}
			}
			\subcaptionbox{A span of cospans morphism}[5cm]{%
				\centering
				\includegraphics{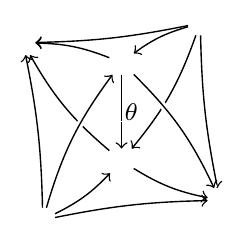}
			}
		\end{minipage}
	}
	\caption{A generic $2$-cell in $\mathbf{Sp}(\mathbf{Csp}(C))$}
	\label{fig:spans_of_cospans}
\end{figure}
However, this version of $\cat{Rewrite}$ does not 
contain everything we need.
In Section \ref{sec:OpenGraphsOverSzx}, 
we fill the gap by introducing  
\emph{open graphs over $S_{\text{zx}}$}.  
That is, we pick a graph $S_{\text{zx}}$
\[
	\includegraphics{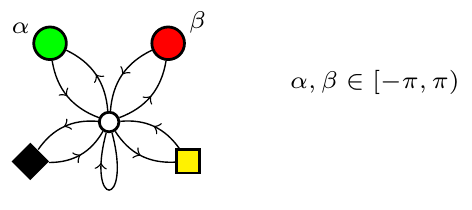}
\]  
whose nodes coincide with the node types 
found in the zx-calculus diagrams, with
one exception: 
the white node in the center.
This node replaces the
dangling edges in the zx-diagrams. 
A graph morphism $G \to S_{\text{zx}}$ 
then corresponds to a zx-morphism
by transporting the node types to $G$
via the fibres of the map.
In his thesis
	\cite{Kissinger_Pictures}, 
Kissinger also colored
graphs this way.  
We then form an SMCC bicategory with 
graphs over $S_{\text{zx}}$ as 0-cells, 
cospans of graphs over $S_{\text{zx}}$ as 1-cells, 
and spans of cospans as 2-cells.  
These spans of cospans are taken
up to the equivalence relation
obtained by relating 2-cells with the
same domain and codomain.
In analogy to the formation of $\cat{Rewrite}$,
we find a sub-bicategory $\mathbf{zxRewrite}$
that can be thought of as containing all 
open graphs over $S_{\text{zx}}$ and their rewrites.

The bicategory $\mathbf{zxRewrite}$ 
is a space in which we can generate 
SMCC sub-bicategories. 
In Section \ref{sec:zx_categorified}, 
we give a presentation for 
a sub-bicategory $\underline{\mathbf{zx}}$ 
of $\mathbf{zxRewrite}$ whose
1-cells correspond to 
zx-calculus diagrams and
2-cells to the relations between them.  
After constructing $\underline{\mathbf{zx}}$, 
we decategorify it to a 1-category 
$|| \underline{ \mathbf{zx} } ||$ 
by identifying 1-cells whenever 
there is a 2-cell between them.  
Though this seems asymmetrical, 
we actually get an equivalence relation
because of the dual nature of spans.  
In the main result, 
Theorem \ref{thm:equiv_of_zx_cats}, 
we construct a dagger compact functor 
$|| \underline{ \mathbf{zx} } || \to \mathbf{zx}$ 
witnessing an equivalence of categories.  
It is in this sense that we are categorifying
the zx-calculus.

The author would like to thank 
John Baez for many helpful ideas and 
discussions that contributed to this paper. 
A debt of gratitude is also owed
to three anonymous referees for their
insightful comments on
an earlier version of this paper
written for the 2017 Quantum Physics and Logic 
conference in Nijmegen, Netherlands.

%
%

\section{The zx-calculus}
\label{sec:ZxCalc}

One of the most fascinating features 
of quantum physics is the 
incompatibility of observables. 
Roughly, an observable is a 
measurable quantity of some system, 
for instance the spin of a photon.  
Incompatibility is in stark contrast 
to classical physics where
measurable quantities are compatible
in that, we can obtain
arbitrarily precise values 
at the same time.   
Arguably, the most famous example 
of incompatibility is 
Heisenberg's uncertainty principal
which places limits to 
the precision that one 
can simultaneously measure a 
pair of observables: 
position and momentum.  
There are different levels of 
incompatibility amongst 
pairs of observables. 
When such a pair is 
maximally incompatible, 
meaning that 
knowing one with complete precision 
implies total uncertainty of the other, 
we say they are 
\emph{complementary observables}.  

Hilbert spaces are, historically, 
the typical framework 
in which one might 
study observables.  
This formalism has been
quite successful
despite involving difficult calculations
and non-intuitive notation. 

The zx-calculus was 
developed by Coecke and Duncan 
	\cite{CoeckeDuncan_QuantumObsFullPaper} 
as a high-level language 
to facilitate such computation
between complementary observables.  
It was immediately used to 
generalize both \emph{quantum circuits}
	\cite{NielsenChuang_QuantumCompInfo}  
and the \emph{measurement calculus}
	\cite{DanosKashefiPanang_MeasurementCalc}. 
Its validity was further justified 
when Duncan and Perdrix 
presented a non-trivial method 
of verifying measurement-based 
quantum computations
	\cite{DuncanPerdrix_RewritingQuantumCompu}.  
At its core, the zx-calculus is 
an intuitive graphical language 
in which to reason about 
complementary observables. 

The five \emph{basic diagrams} 
in the zx-calculus are depicted 
in Figure \ref{fig:ZX_generators}
and are to be read from left to right. 
They are
\begin{itemize}
	\item a \emph{wire} with 
		a single input and output,
	\item \emph{green spiders} with 
		a non-negative integer number of 
		inputs and outputs and 
		paired with a phase $\alpha \in [-\pi,\pi)$,
	\item \emph{red spiders} with 
		a non-negative integer number 
		inputs and outputs and 
		paired with a phase $\beta \in [-\pi,\pi)$,
	\item the \emph{Hadamard node} with
		 a single input and output, and
	\item a \emph{diamond node} with 
		no inputs or outputs.
\end{itemize}
The wire plays the role 
of an identity, much like a 
plain wire in an electrical circuit, 
or straight pipe in a plumbing system. 
The green and red spiders 
arise from a pair of 
complementary observables.  
Incredibly, observables 
correspond to certain 
commutative Frobenius algebras $A$ 
living in a 
dagger symmetric monoidal category 
	$\mathbf{C}$. 
Moreover, a pair of 
complementary observables gives a 
pair of Frobenius algebras whose 
operations interact via laws like 
those of a Hopf algebra 
	\cite{CoeckePavlovic_QuanMeasSums, 
		CoeckePavVicary_OrthogBases}.  
This is particularly nice because 
Frobenius algebras have 
beautiful string diagram representations. 
If $I$ is the monoidal unit of $\mathbf{C}$, 
there is an isomorphism 
	$\mathbf{C}(I,A) \to \mathbf{C}(A,A)$ 
of commutative monoids 
that gives rise to a 
group structure on $A$ 
known as the \emph{phase group}.  
The spider phases arise from this group.
The Hadamard node embodies 
the Hadamard gate. 
The diamond is a scalar obtained when
connecting a green and red node together.  
A deeper exploration of these notions 
goes beyond the scope of this paper.  
For those interested, 
the original paper on the topic
	\cite{CoeckeDuncan_QuantumObsFullPaper}
is an excellent place read more.

In the spirit of compositionality, 
we present a category 
$\mathbf{zx}$ below 
whose morphisms are 
generated by the 
five basic diagrams. 
To anticipate the shift
in terminology, we will
refer to zx-calculus diagrams
as \emph{$\mathbf{zx}$-morphisms} 
and continue to use the
qualifier `\emph{basic}' in the 
same way.

Observe that there is a
non-negative number of 
wires dangling on the
left and right side of each
basic $\mathbf{zx}$-morphism. 
Those on the left, 
we call \emph{inputs} 
and those on the right \emph{outputs}. 
These basic $\mathbf{zx}$-morphisms
generate the morphisms of a
dagger compact category $\mathbf{zx}$
whose objects are the non-negative integers.
This category was introduced 
by Coecke and Duncan 
	\cite{CoeckeDuncan_QuantumObsFullPaper} 
and further studied by Backens 
	\cite{Backens_Completeness}. 
To compose in $\mathbf{zx}$, 
connect compatible diagrams 
along an enumeration of the
the inputs and the outputs. 
A monoidal structure
is given by adding numbers and 
taking the disjoint union of $\mathbf{zx}$-morphisms. 
Relations between the morphisms are
given below, but we note here that
the wire is the identity on $1$.
The identity on $n$ is 
the disjoint union of $n$ wires. 
The symmetry and compactness 
of the monoidal product provide 
a braiding, evaluation, and coevaluation morphisms:
respectively,
\[
\includegraphics{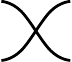}
\quad \quad \quad \quad 
\raisebox{-0.25\height}{%
	\includegraphics[scale=0.75]{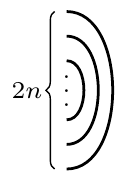}
}
\quad \quad \quad \quad 
\raisebox{-0.25\height}{%
	\includegraphics[scale=0.75]{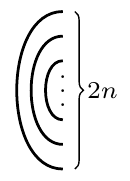}
}
\]
The evaluation and coevalutation maps 
are of type 
	$2n \to 0$ and $0 \to 2n$ 
for each object $n \geq 1$ and 
the empty diagram for $n=0$.  
On the spider diagrams, 
the dagger structure 
swapps inputs and outputs then, 
multiplies the phase by $-1$:
\[
	\includegraphics{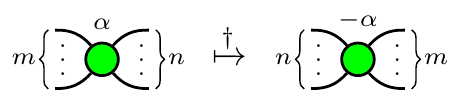}
\]
The dagger acts trivially on the 
wire, Hadamard, and diamond elements. 

\begin{figure}
	\fbox{
		\begin{minipage}{0.97\textwidth}
			\centering
			\subcaptionbox{Spider}{%
				\centering
				\includegraphics{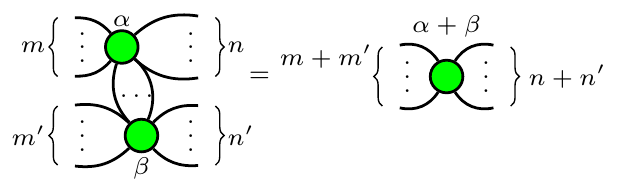}
			}
			\quad \quad
			\subcaptionbox{Bialgebra equation}{%
				\centering
				\includegraphics{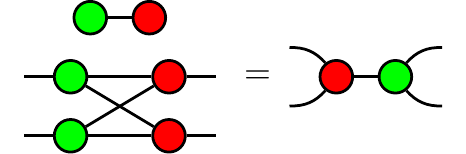}
			}
			\vspace{1em} 
			\linebreak
			\subcaptionbox{Copy equation}{%
				\centering
				\includegraphics{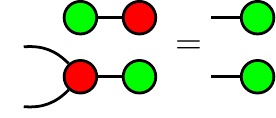}
			}
			\quad \quad 
			\subcaptionbox{$\pi$-Copy equation}{%
				\centering
				\includegraphics{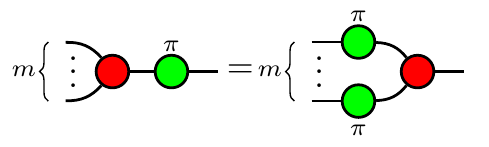}
			}
			\quad \quad 
			\subcaptionbox{Cup equation}[2.5cm]{%
				\centering
				\includegraphics{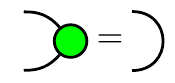}
			}
			\vspace{1em} 
			\linebreak
			\subcaptionbox{Trivial spider equation}[4cm]{%
				\centering
				\includegraphics{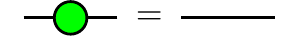}
			}
			\quad \quad \quad \quad \quad \quad 
			\subcaptionbox{$\pi$-Commutation equation}{%
				\centering
				\includegraphics{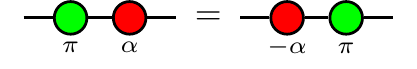}
			}
			\vspace{1em} 
			\linebreak
			\subcaptionbox{Color change equation}{%
				\centering
				\includegraphics{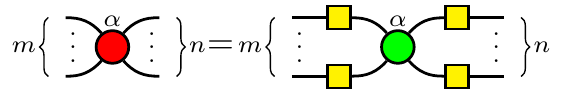}
			}
			\subcaptionbox{Loop equation}[2.5cm]{%
				\centering
				\includegraphics{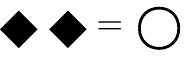}
			}
			\subcaptionbox{Diamond equation}[3.5cm]{%
				\centering
				\includegraphics{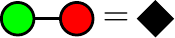}
			}
		\end{minipage}
	}
	\caption{Relations in the category $\mathbf{zx}$}
	\label{fig:ZX_equations}
\end{figure}

Thus far, we have a 
presentation for a 
free dagger compact category. 
However, there are relations 
between $\mathbf{zx}$-morphisms.
These are given in Figure \ref{fig:ZX_equations},  
though we also include equations obtained by 
exchanging red and green nodes, 
daggering, and 
taking diagrams up to 
ambient isotopy in $4$-space. 
These listed relations are called \emph{basic}.  
Spiders with no phase indicated 
have a phase of $0$. 
The emergence of these relations 
goes beyond the scope of this paper
and we point the interested reader to
the genesis of the zx-calculus 
	\cite{CoeckeDuncan_QuantumObsFullPaper}
for an explanation.

A major advantage 
of using string diagrams, 
apart from their intuitive nature, 
is that computations are 
more easily programmed into 
computers.  
Indeed, graphical proof assistants
like Quantomatic 
	\cite{BarKissingerVicary_Globular,
		DixonDuncanKissinger_QuantomaticWebsite} 
and Globular 
	\cite{BarKissingerVicary_Globular} 
were tailor made for such 
graphical reasoning.  
The logic of these programs are 
encapsulated by 
double pushout rewrite rules.  
However, 
the algebraic structure of $\mathbf{zx}$ 
and other graphical calculi 
do not contain the rewrite rules 
as explicit elements.  
Perhaps, 
conceiving of rewrite rules 
as actual elements
in the syntax can prove
beneficial for software programmers.

%
%

\section{Rewriting open graphs}
\label{sec:RewritingOpenGraphs}

Surely, the $\mathbf{zx}$-morphisms 
are reminiscent of directed graphs.  
Hence there is a reasonable optimism
that we can model
the zx-calculus with graphs.  
However, our hope is tempered 
by some clear differences between 
graphs and $\mathbf{zx}$-morphisms.  
For one, graphs do not 
have inputs or outputs. 
In this section, we reconcile 
this particular difference 
by using open graphs. 

Open graphs and 
their morphisms have 
been considered by 
Dixon, Duncan, and Kissinger 
	\cite{Dixon_OpenGraphs}
though our conceit is 
slightly different. 
Conceptually, open graphs 
are quite simple.  
Take a directed graph 
and declare some of the nodes 
to be inputs and 
others to be outputs, 
for example
\[
	\includegraphics{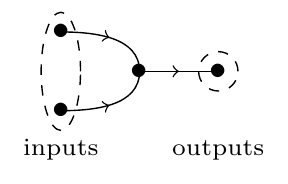}
\]
Given open graphs $G$ and $G'$,
if the set of inputs in $G$
and the set of outputs in $G'$
have the same cardinality,
we can glue them together
along a bijection.
This gives a way 
to turn a pair of 
compatible open graphs
into a single open graph.  
For instance, 
to the above open graph, 
we can connect 
\[
	\includegraphics{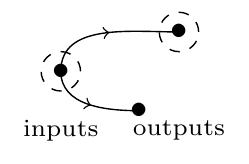}
\]
to form
\[
	\includegraphics{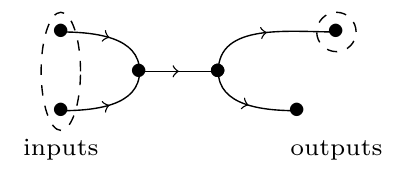}
\]
We make this precise with
cospans and pushouts.

\begin{defn}
	\label{def:Open_Graph}
	Consider the functor 
		$N \colon \mathbf{FinSet}_0 \to \mathbf{Graph}$, 
	on a skeleton of $\mathbf{FinSet}$, 
	defined by the letting 
	$N(X)$ be the edgeless graph 
	with nodes $X$.  
	An \emph{open graph} is then 
	a cospan in the category 
		$\mathbf{Graph}$ 
	of the form 
		$N(X) \to G \gets N(Y)$ 
	for sets $X$ and $Y$.
\end{defn}

The left leg $N(X)$ of the cospan 
gives the \emph{input} and 
the right leg $N(Y)$ the \emph{outputs}.  
Suppose we have another open graph $G'$ 
with inputs $N(Y)$ and outputs $N(Z)$.  
Then we can compose cospans 
\[
	N(X) \to G \gets N(Y) \to G' \gets N(Z). 
\] 
by pushing out over 
$G \gets N(Y) \to G'$ to get 
\[
	N(X) \to G +_{ N ( Y ) } G' \gets N(Z).
\] 
By taking isomorphism classes 
of these pushouts, 
we obtain a category whose 
objects are those in the image of $N$ 
and morphisms are open graphs. 
But we can do better! 

Thus far, we have only 
just described the first layer 
of bicategory
introduced by the author 
under the name $\cat{Rewrite}$
\cite{Cicala_SpansCospans}.
It was shown in a joint work with Courser
\cite{CicalaCourser_BicatSpansCospan} 
that $\mathbf{Rewrite}$ is 
symmetric monoidal and compact closed. 
The monoidal structure is 
induced from the coproduct of graphs.  
Here, we take Stay's definition 
of compact closedness for bicategories
\cite{Stay_CompactClosedBicats}. 
As discussed in the introduction, 
we will work instead with
the slightly modified version of $\cat{Rewrite}$
described in Definition \ref{defn:Rewrite}.

Construction on this modified bicategory
begins with the theorem below.
First, we introduce some needed terminology.
A \emph{span of cospans} is 
a commuting diagram as illustrated 
in Figure \ref{fig:spans_of_cospans}.  
A \emph{parallel class} of spans of cospans is
formed by the equivalence relation 
given by identifying
spans of cospans
with same domain and codomain.

\begin{thm}
	\label{thm:SpCspC_is_SMCC_bicategory}
	Let $\mathbf{C} = (\cat{C}_0,\otimes,I)$ 
	be a finitely complete and cocomplete
	(braided, symmetric) monoidal category 
	such that $\otimes$ preserves colimits.   
	There is a 
	(braided, symmetric) monoidal bicategory 
	$ \spcsp{C} $ 
	whose 
	0-cells are $\cat{C}$-objects, 
	1-cells are cospans in $\cat{C}$, and 
	2-cells are	parallel classes 
	of spans of cospans.
	In case $\cat{C}$ is cocartesian, 
	then $\spcsp{C}$ is also compact closed.
\end{thm}

We prove this theorem in 
Appendix \ref{sec:SMCC_Bicat_SpCsp}. 
It follows from taking parallel classes of 
2-cells that hom-categories 
in $\spcsp{C}$ are groupoids.

Using parallel classes 
of 2-cells instead of
isomorphism classes
has several advantages.  
First, it removes two conditions 
required of $\cat{C}$
in $\bimonspcsp{C}$: 
that $\cat{C}$ be a topos and 
that the legs in the 
span of cospans be monic. 
It also allows us, 
when $\cat{C}$ is sufficiently
like $\cat{Graphs}$, 
to rewrite an edge into a node  
in analogy to deforming a 
topological string 
into a point. 
Moreover, these 2-cells
give us unitary 1-cells.

Recall there are two ways 
to compose 2-cells in a bicategory.  
Horizontal composition in 
$\mathbf{Sp}(\mathbf{Csp}(\mathbf{C}))$ 
uses pushouts and 
vertical composition uses pullbacks:
\begin{equation}
\label{eq:Hor_and_Vert_Composition}
\raisebox{-0.85\height}{%
	\raisebox{0.3\height}{%
		\includegraphics[scale=0.85]{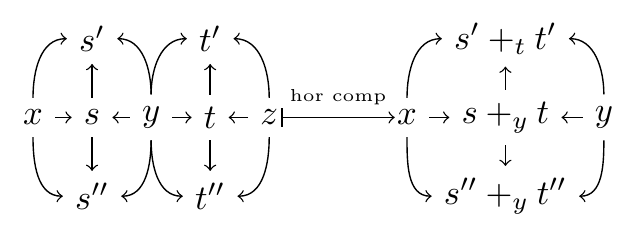} 
	}
	\includegraphics[scale=0.85]{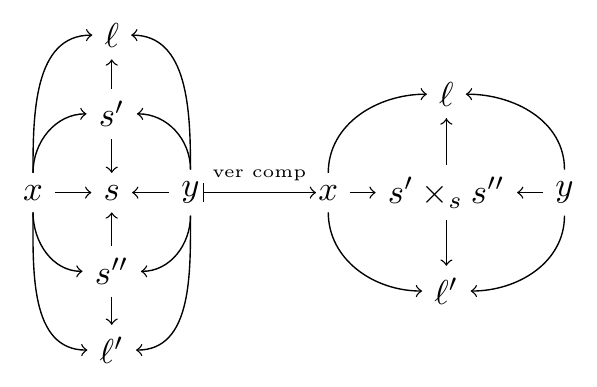}
}
\end{equation}

\begin{defn}
\label{defn:Rewrite}
	Define $\cat{Rewrite}$ to be 
	the 1-full and 2-full 
	SMCC sub-bicategory of 
	$\spcsp{Graph}$ 
	whose 0-cells are 
	exactly those graphs in 
	the image of the functor 
	$N \colon \mathbf{Set}_0 \to \mathbf{Graph}$.
\end{defn}

The conceit of $\mathbf{Rewrite}$ 
is that the 1-cells are open graphs 
whose inputs and outputs are 
chosen by the $0$-cells, 
and the $2$-cells are 
rewrite rules that preserves 
the input and output nodes.
By rewrite rules, we mean those 
taken from the 
double pushout graph rewriting 
approach \cite{Corradini_AlgApprGraphTrans}.
Our open graphs are 
a different formulation
of what amounts to the same concept
explored by Dixon, Duncan, and Kissinger
	\cite{Dixon_OpenGraphs}
though we go a bit further, 
getting an SMCC bicategory
of open graphs instead of a 1-category.
 
Our motivation for constructing $\mathbf{Rewrite}$ 
is not to study it directly, 
but for it to serve 
as an ambient context in which 
to generate SMCC sub-bicategories 
on some collection of open graphs 
and rewriting rules. 
Presenting categories by
open graphs and rewrite rules
is common enough
	\cite{Dixon_OpenGraphs,Fong_AlgOpenSystems,Pollard_OpenMarkov}
to warrant finding a common
framework in which to fit such categories.
However, there are drawbacks to this approach. 
For example, working with open graphs 
is only useful to model graphical languages 
whose terms are equal 
up to ambient isotopy in $4$-space.
This limits the current approach to 
only symmetric monoidal (bi)categories 
as Selinger's work shows
	\cite{Selinger_GraphicsMonCats}.
	
Employing open graphs is 
not quite enough for us
to fully capture the zx-calculus.
We still need to color
our open graphs in a way
that corresponds to the
zx-diagram node types.

%
%

\section{Open graphs over $S_{\text{zx}}$}
\label{sec:OpenGraphsOverSzx}

Last section, 
we began the process of
modeling the zx-calculus with 
graphs by introducing 
open graphs and fitting them 
into a bicategory $\mathbf{Rewrite}$.  
This overcame the issue of 
graphs lacking inputs and outputs.  
In this section, we  
face a different issue.  
Unlike open graphs,
 $\mathbf{zx}$-morphisms 
have multi-sorted nodes.
In this section, we
equip open graphs with
multi-sorted nodes by
working with a 
slice category of 
$\cat{Graph}$. 
Kissinger used a similar
method to give graphs
multi-sorted nodes in 
his thesis
	\cite{Kissinger_Pictures}.

\begin{defn}
	\label{def:graph_over_Szx}
	Let $S$ be a graph.  
	By a \emph{graph over $S$}, 
	we mean a graph morphism $G \to S$. 
	A morphism between graphs over $S$ 
	is a graph morphism $G \to G'$ 
	such that the following
	diagram commutes
	\[
	\begin{tikzpicture}
	\node (1) at (-0.5,1) {$G$};
	\node (2) at (0.5,1) {$G'$};
	\node (3) at (0,0) {$S$};
	\path[->,font=\scriptsize,>=angle 90]
	(1) edge (2)
	(1) edge (3)
	(2) edge (3);
	\end{tikzpicture}
	\]
\end{defn} 

Every graph morphism $G \to G'$ contains 
a map between corresponding nodes sets.
The fibre of this map
colors the $G$-nodes 
with the $G'$-nodes.
We illustrate this with the following example.

\begin{ex}
	\label{ex:basic_graph_over_Szx}
	Let $S_{\text{zx}}$ be the graph
	\begin{equation}
	\label{diag:zx_struture_graph}
	\raisebox{-0.75\height}{
		\includegraphics{InclGrphx--slice--graph_S_zx}
	}
	\end{equation}
	We have not drawn 
	the entirety of $S_{\text{zx}}$. 
	The green and red nodes actually
	run through $[-\pi,\pi)$ and 
	all of them have 
	a single arrow to and from node 
	$
	\begin{tikzpicture}
		\node [zxopen] at (0,0) {};
	\end{tikzpicture}
	$. 
	
	Most of the structure of the 
	basic $\mathbf{zx}$-morphisms 
	is captured by graphs over $S_{\text{zx}}$.
	Consider the following graphs over $S_{\text{zx}}$
	\[
	\label{fig:ZX_1cells_generators}
		\begin{minipage}{0.33\textwidth}
			\centering
			\includegraphics{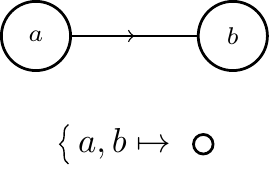}
		\end{minipage}
		\begin{minipage}{0.34\textwidth}
			\centering
			\includegraphics{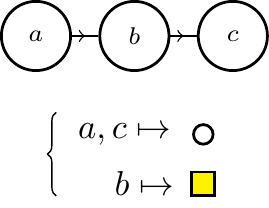}
		\end{minipage}
		\begin{minipage}{0.33\textwidth}
			\centering
			\includegraphics{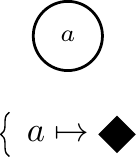}
		\end{minipage}
	\]
	\[
	\begin{minipage}{0.5\textwidth}
		\centering
		\includegraphics{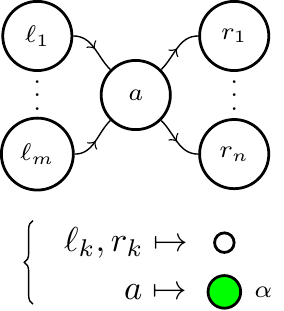}
	\end{minipage}
	\begin{minipage}{0.5\textwidth}
		\centering
		\includegraphics{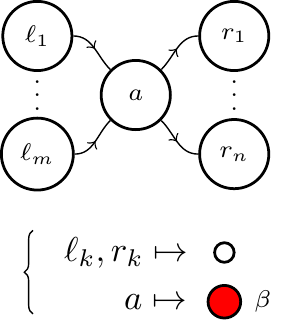}
	\end{minipage}
	\]
	where the diagrams give the 
	domain of each graph over $S_{\text{zx}}$
	and the map is described directly
	underneath each diagram.
	The behavior of each map
	is determined by the image
	of the nodes
	because there is at most 
	one arrow between any 
	two nodes in $S_{\text{zx}}$.  
	The role played by each node 
	of $S_{\text{zx}}$ in 
	providing our desired structure 
	is evident except, perhaps, 
	for the node
	$
	\begin{tikzpicture}
		\node [zxopen] at (0,0) {};
	\end{tikzpicture}
	$.   
	Observe that four of the 
	basic $\mathbf{zx}$-morphisms 
	have dangling wires on either end.  
	Because edges of directed graphs 
	must be attached to 
	a pair of nodes, 
	we use this node
	to anchor the dangling edges.
\end{ex}

\begin{ex}
	\label{ex:graph_over_Szx}
	At this point, 
	we can interpret
	the basic zx-diagrams as 
	graphs over $S_{\text{zx}}$. 
	This extends nicely to 
	a translation of any
	$\mathbf{zx}$-morphism, 
	such as
	\[
	\includegraphics{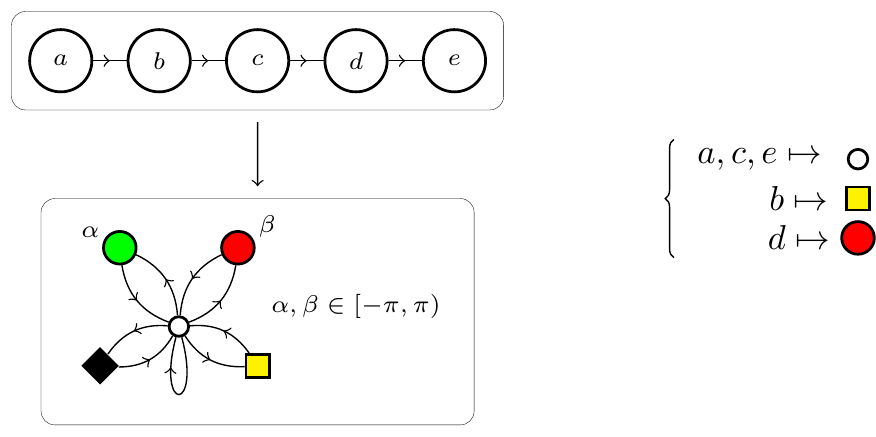}
	\]
	which corresponds to the $\mathbf{zx}$-morphism 
	\[
	\includegraphics{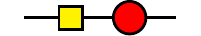}
	\]
\end{ex}

The graphs over $S_{\text{zx}}$ 
in Examples 
\ref{ex:basic_graph_over_Szx} and
\ref{ex:graph_over_Szx}
capture most of the structure 
of the basic $\mathbf{zx}$-morphisms.  
The ability to compose is still missing.
Composition becomes possible
with \emph{open graphs over $S_{\text{zx}}$}.
Again, we use cospans 
to make this precise, 
though combining these two structure 
introduces new considerations. 

Start with the slice category 
$\mathbf{Graph} \downarrow S_{\text{zx}}$ 
of graphs over $S_{\text{zx}}$. 
By Theorem 
	\ref{thm:SpCspC_is_SMCC_bicategory}, 
this gives us an SMCC bicategory 
$\mathbf{Sp} ( \mathbf{Csp} 
	( \mathbf{Graph} \downarrow S_{\text{zx}} ) ) $  
within which we want to 
construct a sub-bicategory 
analogous to $\mathbf{Rewrite}$.  
However, there is a problem. 
Recall that the objects of $\mathbf{Rewrite}$ 
have form $N(X)$ where 
	$N \colon \mathbf{Set}_0 \to \mathbf{Graph}$ 
is the functor sending a set 
to the edgeless graph on that set.  
In $\mathbf{Graph}$, 
there is a unique, 
up to isomorphism,
way to be an edgeless graph.
But in $\mathbf{Graph} \downarrow S_{\text{zx}}$, 
there may be many ways
to be edgeless.
This depends on the number of
graph morphisms to $S_{\text{zx}}$. 
For instance, a graph with 
$n$ nodes and no edges 
can be a graph over $S_{\text{zx}}$ 
in $5^n$ ways. 
We rectify this issue by functorially
turning a set into an edgeless graph. 
Recall that $\mathbf{FinSet}_0$ is a
skeleton of $\mathbf{FinSet}$.

\begin{defn}
	\label{def:Nzx_functor}
	Define a functor 
		$N_{\text{zx}} \colon \mathbf{FinSet}_0 
			\to 
			\mathbf{Graph} \downarrow S_{\text{zx}}$ 
	by 
		\[
			X \mapsto (N_{\text{zx}}(X) \to S_{\text{zx}})
		\]
	where $N_{\text{zx}}(X)$ is 
	the edgeless graph with 
	nodes $X$ that are 	
	constant over 
	$\begin{tikzpicture} 
		\node [zxopen] at (0,0) {}; 
	\end{tikzpicture}$. 
	An \emph{open graph over $S_{\text{zx}}$} 
	is a cospan in 
	$\mathbf{Graph} \downarrow S_{\text{zx}}$ 
	of the form
	\[
		N_{\text{zx}}(X) \to G \gets N_{\text{zx}} (Y).
	\]
\end{defn}

With this definition of
open graphs over $S_{\text{zx}}$, 
we propose the analogue 
to $\mathbf{Rewrite}$. 

\begin{defn}
	Define $\mathbf{zxRewrite}$ 
	to be the symmetric monoidal and 
	compact closed sub-bicategory of 
	$\mathbf{Sp} ( \mathbf{Csp} 
	( \mathbf{Graph} \downarrow S_{\text{zx}} ) ) $  
	that is 1-full and 2-full 
	on objects of the form 
	$N_{\text{zx}} (X)$ for finite sets $X$.  
\end{defn}

Unpacking this definition, 
the 0-cells of $\mathbf{zxRewrite}$ 
are those edgeless graphs over $S_{\text{zx}}$ 
in the image of $N_{\text{zx}}$.  
The 1-cells are exactly 
the open graphs over $S_{\text{zx}}$. 
The 2-cells are the rewritings of
one open graph over $S_{\text{zx}}$ 
into that preserve the inputs and outputs.  
To better understand this bicategory, 
we give an example of 
an open graph over $S_{\text{zx}}$. 
Along with this example, 
we present a new notation 
that allow us to draw the
remaining diagrams 
more compactly.

\begin{ex}
	\label{ex:open_graph_over_Szx}
	Consider the graph over $S_{\text{zx}}$ 
	in Example \ref{ex:graph_over_Szx}.  
	Make this an open graph as follows:
	\[
		\includegraphics[scale=0.9]{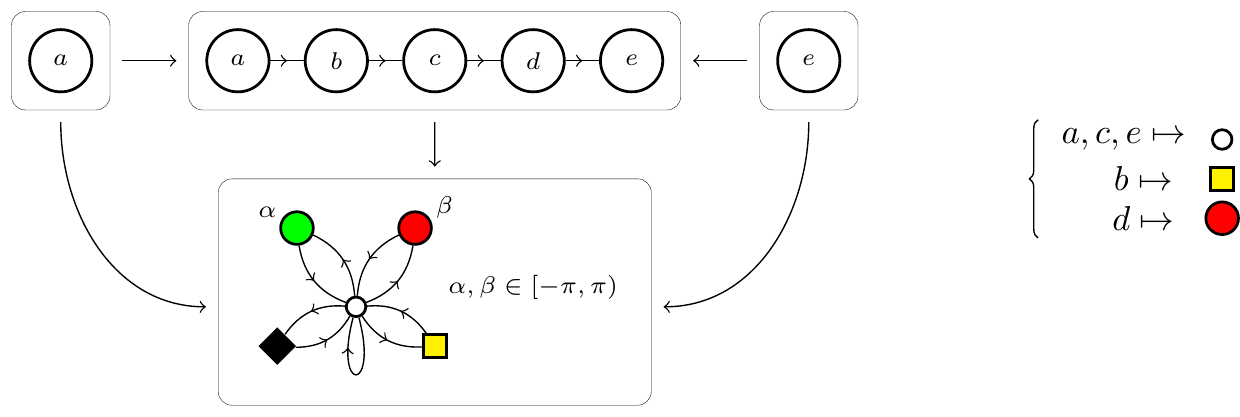}
	\]
	There is a single input, 
	node $a$, and 
	a single output, node $e$. 
	Denote this by
	\[
		\includegraphics{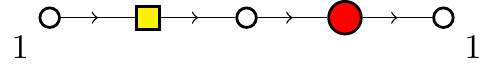}
	\]
	The input nodes are 
	aligned on the far left and 
	the output nodes on the far right.  
	The $1$'s in the corners 
	refer to the cardinality of 
	the input and output node sets.  
	This may seem unnecessary 
	or redundant, 
	but it will clarify 
	several situations arising later on. 
	Thus, we side with consistency 
	and always write the cardinalities.
	Of course, this notation 
	strips a fair amount information 
	regarding the graph morphisms 
	involved in the cospan. 
	However, any missing information 
	should be evident in context.
\end{ex}

Recall that
our interest in 
$\mathbf{Rewrite}$ 
is as an ambient space
in which to generate 
syntactical bicategories 
for graphical languages.
The same is true of
our new bicategory $\mathbf{zxRewrite}$.  
Presently, we are interested in 
1-cells corresponding to 
the basic $\mathbf{zx}$-morphisms 
and 2-cells to the basic relations 
depicted in Figure \ref{fig:ZX_equations}.  
We claim that the
bicategory generated by 
these 1-cells and 2-cells 
categorifies $\mathbf{zx}$.
 
%
%

\section{A categorification of $\mathbf{zx}$}
\label{sec:zx_categorified}

Section \ref{sec:OpenGraphsOverSzx} 
describes a translation of
the basic $\mathbf{zx}$-morphisms 
into open graphs over $S_{\text{zx}}$.
These are depicted in 
Figure \ref{fig:ZX_1cells_generators}
and are referred to as 
\emph{basic open graphs over $S_{\text{zx}}$}.  
To clarify the double instances of $m$ and $n$ 
written in the spider diagrams, 
those below the diagram 
refer to the cardinalities of the cospan legs, 
and those beside the brackets 
count how many nodes are there
(cf.\ Example \ref{ex:open_graph_over_Szx}).  

\begin{figure}[h]
	\fbox{%
		\begin{minipage}{0.975\textwidth}
			\centering
			\subcaptionbox{Wire}{%
				\centering
				\includegraphics{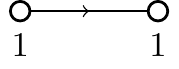}
			}
			\quad \quad \quad \quad
			\subcaptionbox{Hadamard}[2.5cm]{%
				\centering
				\includegraphics[]{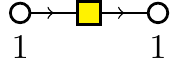}
			}
			\quad \quad \quad \quad
			\subcaptionbox{Diamond}[2.5cm]{%
				\centering
				\includegraphics{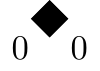}
			}
			\vspace{1em}
			\linebreak
			\subcaptionbox{Green spider}{%
				\centering
				\includegraphics{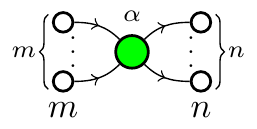}
			}
			\quad \quad \quad \quad
			\subcaptionbox{Red spider}{%
				\centering
				\includegraphics{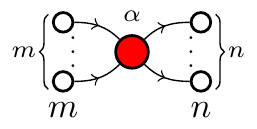}
			}
		\end{minipage}
	}
	\caption{
		Generating $1$-cells for the bicategory $\underline{\mathbf{zx}}$ }
	\label{fig:ZX 1cells_generators}
\end{figure} 

Just as the basic open graphs over $S_{\text{zx}}$ 
capture the generating $\mathbf{zx}$-morphisms, 
we must also include
the basic relations into our framework.  
Figure 
	\ref{fig:ZX_2cells_generators}
depicts our representation
of the basic relations
as spans of open graphs. 
In addition to the basic relations
listed explicitly, we add 
those obtained by
exchanging red and green nodes, 
swapping inputs and outputs, 
turning the spans around,
as well as 
\begin{equation}
\label{eq:2cell_wire_is_identity}
\raisebox{-2.5em}{%
	\includegraphics{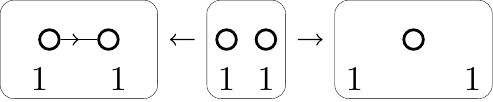}
}
\end{equation}
The last is added to 
ensure that the wire 1-cell 
behaves as the identity, 
a property we lose when 
using graphs.
All of these 2-cells, 
we call \emph{basic}.  

It is important to emphasize 
that the basic 2-cells are 
representatives of an equivalence class. 
That is, we have made a decision 
to present these 2-cells 
as a span of cospans 
whose apex is the 
edgeless graph whose 
node set is the disjoint union 
of the inputs and outputs.  
This is certainly not 
the only representative 
we could have chosen,
though it does seem to be 
the most natural choice. 

Forget for a moment that 
our 2-cells are classes and 
think of only the representatives.  
When we compose 
a span of cospans with its dagger, 
we get a non-trivial way to 
rewrite a 1-cell into itself.  
This ought to be 
distinct from the identity rewrite, 
which does nothing.  
However, our choice of 
equivalence classes 
render these the same.
This hints that a higher rewriting structure 
is hiding in the background.  
Indeed, conveying rewrite rules 
as spans of cospans has the 
advantage of including 
higher level rewrite rules in 
by iterating the process of taking spans. 
Currently, we content ourselves
to work within bicategories and 
leave an exploration for
higher structure for another time. 

We now define 
the bicategory which 
categorifies the zx-calculus.

\begin{figure}[h]
	\fbox{%
		\begin{minipage}{0.975\textwidth}
			\centering
			\subcaptionbox{Spider}[\textwidth]{%
				\includegraphics[scale=0.75]{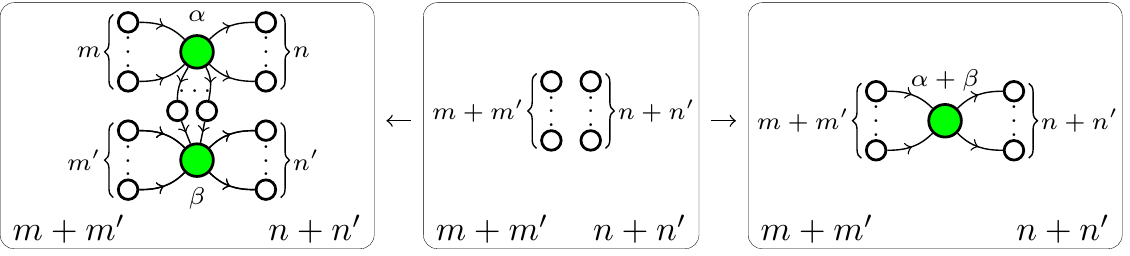}
			}
			\linebreak
			\subcaptionbox{Bialgebra}[\textwidth]{%
				\includegraphics[scale=0.75]{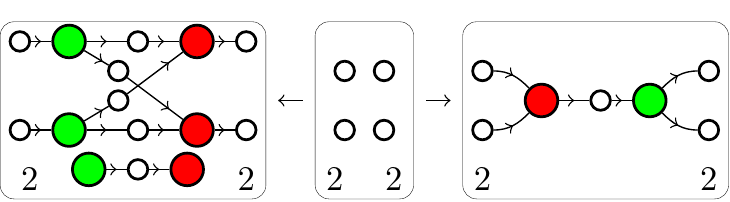}
			}
			\linebreak
			\subcaptionbox{Cup}[0.45\textwidth]{%
				\includegraphics[scale=0.75]{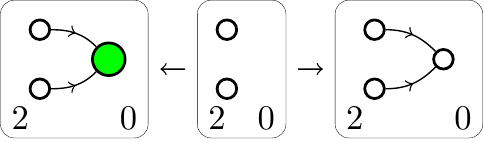}
			}
			\subcaptionbox{Copy}[0.45\textwidth]{%
				\includegraphics[scale=0.75]{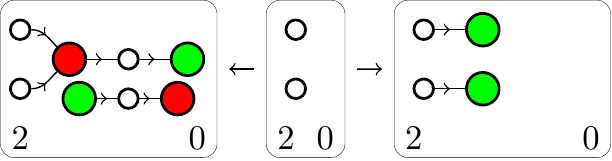}
			}
			\linebreak
			\subcaptionbox{Trivial spider}[0.45\textwidth]{%
				\includegraphics[scale=0.75]{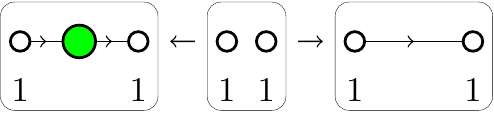}
			}
			\subcaptionbox{$\pi$-copy}[0.45\textwidth]{%
				\includegraphics[scale=0.75]{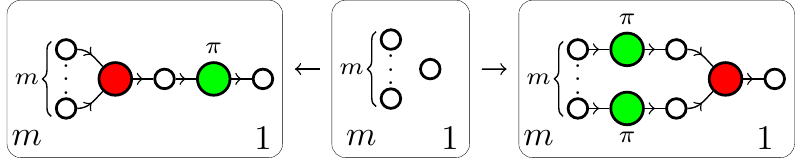}
			}
			\linebreak
			\subcaptionbox{$\pi$-commutation}[\textwidth]{%
				\includegraphics[scale=0.75]{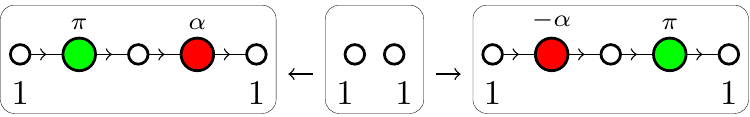}
			}
			\linebreak
			\subcaptionbox{Color change}[\textwidth]{%
				\includegraphics[scale=0.75]{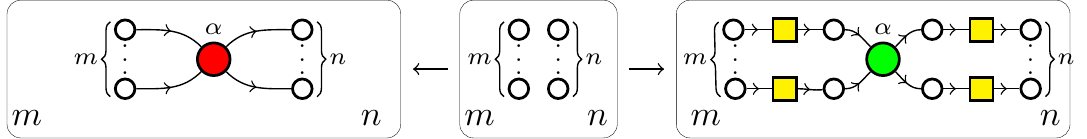}
			}
			\linebreak
			\subcaptionbox{Loop}[0.45\textwidth]{%
				\includegraphics[scale=0.75]{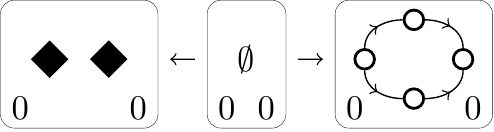}
			}
			\subcaptionbox{Diamond}[0.45\textwidth]{%
				\includegraphics[scale=0.75]{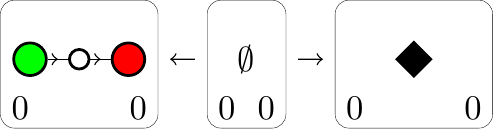}
			}
		\end{minipage}
	}
	\caption{Generating $2$-cells for the bicategory $\underline{\mathbf{zx}}$}
	\label{fig:ZX_2cells_generators}
\end{figure}

\begin{defn}
	\label{def:zx_bicat}
	Define $\underline{\mathbf{zx}}$ to be 
	the symmetric monoidal and 
	compact closed sub-bicategory 
	of $\mathbf{zxRewrite}$ generated by 
	the basic $1$-cells and basic $2$-cells.
\end{defn}

Working within $\cat{zxRewrite}$
allows us to generate 
$\underline{\mathbf{zx}}$
as an SMCC in this way.
Without having this ambient space,
we cannot be sure that we 
obtain an SMCC bicategory simply
by giving a presentation.

Because $\underline{\mathbf{zx}}$ is 
symmetric monoidal and compact closed, 
it contains twist, 
evaluation, and coevaluation 1-cells 
\begin{equation}
\label{diag:TwistCompact_1cells}
\raisebox{-0.8\height}{
	\includegraphics[scale=0.75]{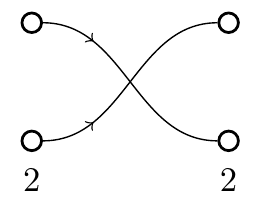}
	\quad \quad \quad
	\includegraphics[scale=0.75]{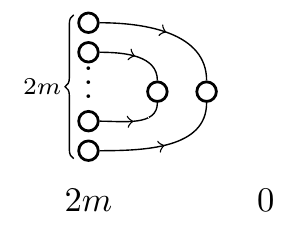}
	\quad \quad \quad
	\includegraphics[scale=0.75]{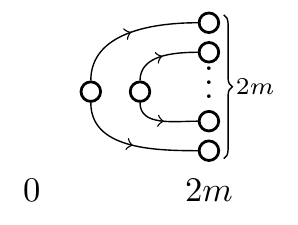}
}
\end{equation}
witnessing symmetry and 
the compact structure on $m$.

Horizontal and vertical composition 
are the same as in
	\ref{eq:Hor_and_Vert_Composition}. 
For example, 
we can compose spider diagrams 
with the same number 
of inputs and outputs:
\[
	\includegraphics{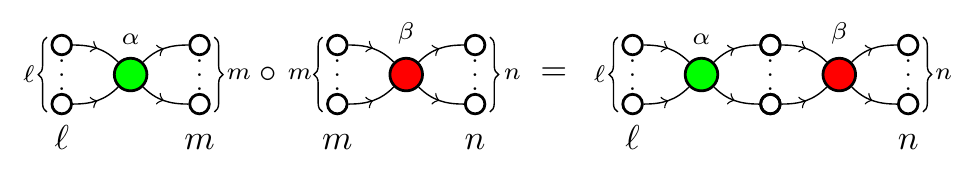}
\]
We tensor 1-cells by 
disjoint union, such as
\[
	\includegraphics{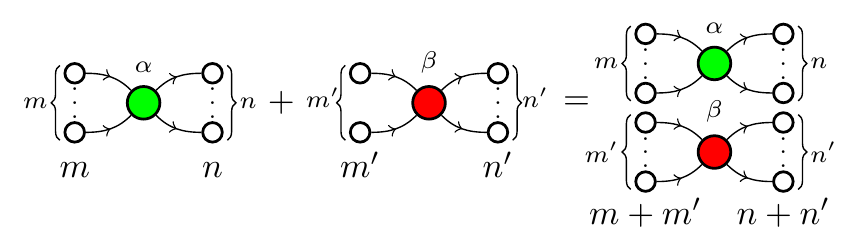}
\]
With $\underline{\mathbf{zx}}$ defined, 
we turn our focus towards 
presenting the main theorem. 
We start by giving a category that 
is a \emph{decategorification} 
or \emph{truncation} of $\underline{\mathbf{zx}}$.  

\begin{defn}
	\label{def:decat_zx}
	Define $|| \underline{\mathbf{zx}} ||$ 
	to be the category whose 
	objects are the 0-cells of $\underline{\mathbf{zx}}$ and 
	whose arrows are the 1-cells of $\underline{\mathbf{zx}}$ 
	modulo the equivalence relation $\sim$ 
	given by: 
	$f \sim g$ if and only if 
	there is a 2-cell $f \Rightarrow g$ 
	in $\underline{\mathbf{zx}}$.
\end{defn}

To be clear, $\sim$ 
is an equivalence relation 
and doesn't merely generate one.  
This follows from 
the symmetry of spans 
and vertical composition.   

\begin{thm}
	\label{thm:decat_zx_is_dagger_compact}
	The category $|| \underline{\mathbf{zx}} ||$ 
	is dagger compact via 
	the identity-on-objects functor $\dagger$
	given by 
	\[
	\begin{minipage}{0.4\textwidth}
		\includegraphics{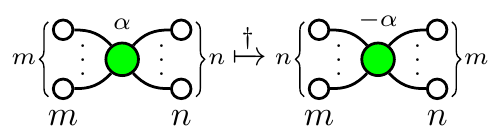}
	\end{minipage}
	\begin{minipage}{0.075\textwidth}
		\text{ and } 
	\end{minipage}
	\begin{minipage}{0.4\textwidth}
		\includegraphics{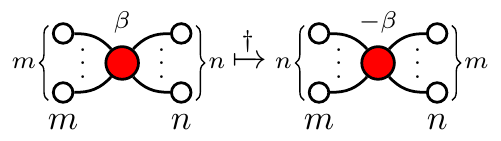}
	\end{minipage}
	\]
	as well as by 
	identity on the wire, 
	Hadamard, and diamond morphisms.
\end{thm}
\begin{proof}
	Compact closedness 
	follows from the self duality of objects 
	via the evaluation and coevaluation 
	maps from \eqref{diag:TwistCompact_1cells}. 
	The snake equation is derived by
	\[
	\includegraphics[scale=0.75]{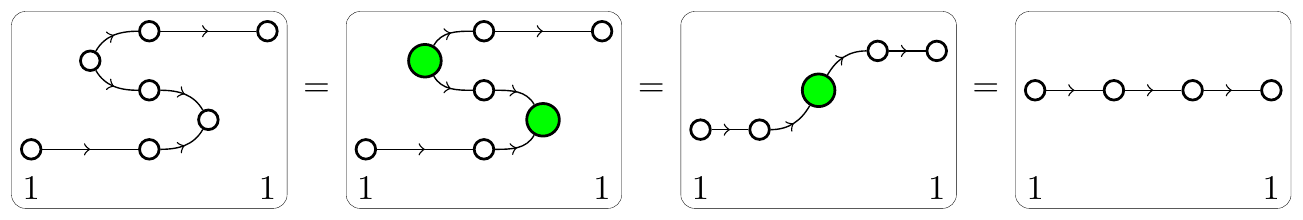}
	\]
	where the equalities 
	follow from the evident 
	2-cells in $\underline{\mathbf{zx}}$. 
	The extra relation 
		\eqref{eq:2cell_wire_is_identity} 
	ensures that the 
	string of wires is the identity.  
	Showing that $\dagger$
	is a dagger functor is a matter of 
	checking some easily verified details. 
\end{proof}

\begin{thm}
	\label{thm:equiv_of_zx_cats}
	The identity on objects, 
	dagger compact functor 
		$E \colon \mathbf{zx} \to || \underline{\mathbf{zx}} ||$ 
	given by
	\[
	\begin{minipage}{0.175\textwidth}
		\includegraphics{InclGrphx--generater--green_spider}
	\end{minipage}
		\mapsto
	\begin{minipage}{0.2\textwidth}
		\includegraphics{InclGrphx--1cell--zx_green_spider}
	\end{minipage}
	\quad \quad \quad 
	\begin{minipage}{0.175\textwidth}
		\includegraphics{InclGrphx--generater--red_spider}
	\end{minipage}
	\mapsto
	\begin{minipage}{0.2\textwidth}
		\includegraphics{InclGrphx--1cell--zx_red_spider}
	\end{minipage}
	\]
	\[
	\begin{minipage}{0.1\textwidth}
		\includegraphics{InclGrphx--generater--wire}
	\end{minipage}
	\mapsto
	\begin{minipage}{0.15\textwidth}
		\includegraphics{InclGrphx--1cell--zx_wire}
	\end{minipage}
	\quad \quad \quad 
	\begin{minipage}{0.04\textwidth}
		\includegraphics{InclGrphx--generater--diamond}
	\end{minipage}
	\mapsto
	\begin{minipage}{0.06\textwidth}
		\includegraphics{InclGrphx--1cell--zx_diamond}
	\end{minipage}
	\quad \quad \quad 
	\begin{minipage}{0.1\textwidth}
		\includegraphics{InclGrphx--generater--hadamard}
	\end{minipage}
	\mapsto
	\begin{minipage}{0.1\textwidth}
		\includegraphics{InclGrphx--1cell--zx_hadamard}
	\end{minipage}
	\]
	is an equivalence of categories.
\end{thm}
\begin{proof}
	That $E$ is identity-on-objects
	implies essential surjectivity.  
	Fullness holds because 
	the generating morphisms for 
	$|| \underline{\mathbf{zx}} ||$  
	are all in the image of $E$. 
	
	Proving faithfulness is more involved. 
	Let $f$ and $g$ be $\mathbf{zx}$-morphisms. 
	Consider representatives
	$\widetilde{Ef}$, $\widetilde{Eg}$
	of $Ef$, $Eg$ obtained by 
	translating directly
	the graphical representation of $f,g$ 
	to open graphs over $S_{\text{zx}}$ 
	as in Examples \ref{ex:graph_over_Szx}  
	and \ref{ex:open_graph_over_Szx}. 
	It suffices to show that the 
	existence of a 2-cell 
		$\widetilde{Ef} \Rightarrow \widetilde{Eg}$ 
	in $\underline{\mathbf{zx}}$ 
	implies that $f=g$.  
	
	Observe that any 
	2-cell $\alpha$ in 
	$\underline{\mathbf{zx}}$ 
	can be written, 
	not necessarily uniquely, 
	as a sequence 
		$\alpha_1 \square \dotsm \square \alpha_n$ 
	where each $\alpha_i$ is a basic 2-cell, 
	each box is filled with 
	`$\circ_\text{h}$', `$\circ_\text{v}$', or `$+$', 
	and parentheses are right justified.  
	By `$\circ_\text{h}$' and `$\circ_\text{v}$', 
	we mean horizontal and vertical composition. 
	We induct on sequence length.  
	If $\alpha \colon \widetilde{Ef} \Rightarrow \widetilde{Eg}$ 
	is a basic 2-cell, 
	then there is clearly 
	a corresponding basic relation 
	equating $f$ and $g$.  
	Suppose we have a sequence 
	of length $n+1$ such that 
	the left-most square is a `$+$'. 
	Then we have a 2-cell 
	$\alpha_1 + \alpha_2 \colon Ef \Rightarrow Eg$ 
	where $\alpha_1$ is a basic 2-cell 
	and $\alpha_2$ can be 
	written with length $n$.   
	By fullness, we can write 
	$\alpha_1 + \alpha_2 \colon Ef_1 + EF_2 \Rightarrow Eg_1 + Eg_2$ 
	where $\alpha_i \colon Ef_i \Rightarrow Eg_i$.   
	This gives that $f_i = g_i$  
	and the result follows.  
	A similar argument handles 
	the cases when the 
	left-most operation is 
	vertical or horizontal composition.
\end{proof}
 
%
%

\section{Conclusion}
\label{sec:Conclusion}

The main advantage of 
fitting the zx-calculus 
into a bicategory is that 
the rewrite rules are now 
explicitly included
into the mathematical structure.
That is, we are now capturing
a larger portion
of the full picture that is
the zx-calculus. 

However, categorifying the
zx-calculus is only
part of the story.
The methods used here
are general
with only slight tweaks
made to accommodate the
case at hand.  
Indeed, similarly to how
we use $\cat{zxRewrite}$
to capture zx-diagrams,
we can construct modified versions
of $\cat{Rewrite}$ to frame 
open Markov chains, resistor networks,
internal Frobenius algebras, etc
into this framework.

\appendix
%
%

\section{A symmetric monoidal bicategory of spans of cospans}
\label{sec:SMCC_Bicat_SpCsp}

In this section, 
we prove Theorem 
	\ref{thm:SpCspC_is_SMCC_bicategory}. 
Let $\cat{C} = (\cat{C}_0, \otimes, I)$ be 
a finitely complete and cocomplete 
(braided, symmetric) monoidal category 
such that $\otimes$ preserves colimits. 
The category $\cat{Graph}$ together
with its coproduct is one example.
The proof consists of two parts: 
that $\spcsp{C}$ is a (braided, symmetric) 
monoidal bicategory and 
that it is compact closed.

We first show $\spcsp{C}$ 
is a (braided, symmetric) 
monoidal bicategory 
with a result from Shulman. 

\begin{thm}{\cite[Theorem 5.1]{Shulman_ConstructSMBicats}}
	\label{thm:DoubleGivesBi}
	Let $\dblcat{D}$ be an 
	isofibrant (braided, symmetric) 
	monoidal double category. 
	There is a (braided, symmetric) 
	monoidal bicategory $\cat{D}$ 
	whose objects are those of $\dblcat{D}$ 
	and whose hom-categories $\cat{D}(x,y)$ 
	have as 
	objects the horizontal arrows 
	in $\dblcat{D}$ of type $x \to y$ 
	and as morphisms the 
	2-cells in $\dblcat{D}$ of type
	\[
	\begin{tikzpicture}
		\node (1) at (0,1) {$x$};
		\node (2) at (1,1) {$y$};
		\node (3) at (1,0) {$y$};
		\node (4) at (0,0) {$x$};
		\node at (0.5,0.5) {$\Downarrow$};
		\draw [->] (1) to (2);
		\draw [->] (1) to node[left] {\scriptsize id} (4);
		\draw [->] (2) to node[right] {\scriptsize id} (3);
		\draw [->] (4) to (3);
	\end{tikzpicture}
	\]
\end{thm}

The same paper 
	\cite{Shulman_ConstructSMBicats} 
also contains the definitions used in this section.  
To use this theorem, we begin construction 
on a double category $\dblspcsp{C}$. 
This requires \emph{cubical spans of cospans}, 
which are commuting diagrams of shape 
\[
\begin{tikzpicture}
	\node (A) at (0,2) {$\bullet$};
	\node (A') at (0,1) {$\bullet$};
	\node (A'') at (0,0) {$\bullet$};
	\node (B) at (1,2) {$\bullet$};
	\node (B') at (1,1) {$\bullet$};
	\node (B'') at (1,0) {$\bullet$};
	\node (C) at (2,2) {$\bullet$};
	\node (C') at (2,1) {$\bullet$};
	\node (C'') at (2,0) {$\bullet$};
	\path[->,font=\scriptsize,>=angle 90]
	(A) edge node{} (B)
	(A') edge node{} (B')
	(A'') edge node{} (B'')
	(C) edge node{} (B)
	(C') edge node{} (B')
	(C'') edge node{} (B'')
	(A') edge node{} (A)
	(A') edge node{} (A'')
	(B') edge node{} (B)
	(B') edge node{} (B'')
	(C') edge node{} (C)
	(C') edge node{} (C'');
\end{tikzpicture}
\]
We get an equivalence relation on these 
by relating cubical spans of cospans 
that share the same outside square.
The induced classes are called
\emph{parallel classes}.

\begin{lem}
	\label{lem:SpanCospanDoubleCat}
	There is a double category $\dblspcsp{C}$ 
	whose objects are the $\cat{C}$-objects, 
	vertical morphisms are given by 
	isomorphism classes of spans in $\cat{C}$ with invertible legs, 
	horizontal morphisms are given by cospans in $\cat{C}$, 
	and 2-morphisms are parallel classes of 
	cubical spans of cospans in $\cat{C}$. 
\end{lem}

\begin{proof}
	Define the object category 
	$\widetilde{\dblcat{C}}_0$ 
	to have as objects 
	the $\cat{C}$-objects and 
	as morphisms the isomorphism classes of 
	spans in $\cat{C}$ with invertible legs. 
	Define the arrow category 
	$\widetilde{\dblcat{C}}_1$ 
	to have as objects the 
	cospans in $\cat{C}$ and 
	as morphisms the 
	parallel classes of 
	cubical spans of cospans 
	in $\cat{C}$.  
	
	The structure functor 
		$U \from \widetilde{\dblcat{C}}_0 \to \widetilde{\dblcat{C}}_1$ 
	acts on objects by 
	mapping $x$ to the identity cospan on $x$ 
	and on morphisms by 
	mapping $x \gets y \to z$, whose legs are isomorphisms, 
	to 
	\[
	\begin{tikzpicture}
	\node (A) at (0,2) {$x$};
	\node (A') at (0,1) {$x$};
	\node (A'') at (0,0) {$x$};
	\node (B) at (1,2) {$y$};
	\node (B') at (1,1) {$y$};
	\node (B'') at (1,0) {$y$};
	\node (C) at (2,2) {$z$};
	\node (C') at (2,1) {$z$};
	\node (C'') at (2,0) {$z$};
	\path[->,font=\scriptsize,>=angle 90]
	(A) edge node{} (B)
	(A') edge node{} (B')
	(A'') edge node{} (B'')
	(C) edge node{} (B)
	(C') edge node{} (B')
	(C'') edge node{} (B'')
	(A') edge node{} (A)
	(A') edge node{} (A'')
	(B') edge node{} (B)
	(B') edge node{} (B'')
	(C') edge node{} (C)
	(C') edge node{} (C'');
	\end{tikzpicture}
	\]
	The source functor 
		$S \from \widetilde{\dblcat{C}}_1 \to \widetilde{\dblcat{C}}_0$ 
	acts on objects by sending 
	$x \to y \gets z$ to $x$ and 
	on morphisms by sending 
	a cubical span of cospans to the span 
	occupying the left vertical side.  
	The target functor $T$ is defined similarly.  
	
	The horizontal composition functor 
		$\odot \from 
			\widetilde{\dblcat{C}}_1 \times_{\widetilde{\dblcat{C}}_0} \widetilde{\dblcat{C}}_1 \to \widetilde{\dblcat{C}}_1$ 
	acts on objects by 
	composing cospans with pushouts 
	in the usual way.  
	It acts on morphisms by 
	\[
	\raisebox{-0.5\height}{
		\begin{tikzpicture}
		\node (A) at (0,2) {$a$};
		\node (A') at (0,1) {$a'$};
		\node (A'') at (0,0) {$a''$};
		\node (B) at (1,2) {$b$};
		\node (B') at (1,1) {$b'$};
		\node (B'') at (1,0) {$b''$};
		\node (C) at (2,2) {$c$};
		\node (C') at (2,1) {$c'$};
		\node (C'') at (2,0) {$c''$};
		\node (D) at (3,2) {$d$};
		\node (D') at (3,1) {$d'$};
		\node (D'') at (3,0) {$d''$};
		\node (E) at (4,2) {$e$};
		\node (E') at (4,1) {$e'$};
		\node (E'') at (4,0) {$e''$};
		\path[->,font=\scriptsize,>=angle 90]
		(A) edge node[above]{} (B)
		(A') edge node[above]{} (B')
		(A'') edge node[above]{} (B'')
		(C) edge node[above]{} (B)
		(C') edge node[above]{} (B')
		(C'') edge node[above]{} (B'')
		(C) edge node[above]{} (D)
		(C') edge node[above]{} (D')
		(C'') edge node[above]{} (D'')
		(E) edge node[above]{} (D)
		(E') edge node[above]{} (D')
		(E'') edge node[above]{} (D'')
		(A') edge node[left]{} (A)
		(A') edge node[left]{} (A'')
		(B') edge node[left]{} (B)
		(B') edge node[left]{} (B'')
		(C') edge node[left]{} (C)
		(C') edge node[left]{} (C'')	
		(D') edge node[left]{} (D)
		(D') edge node[left]{} (D'')
		(E') edge node[left]{} (E)
		(E') edge node[left]{} (E'');
		\end{tikzpicture}
	}
	\quad
	\xmapsto[]{\odot}
	\quad
	\raisebox{-0.5\height}{
		\begin{tikzpicture}
		\node (A) at (0,2) {$a$};
		\node (A') at (0,1) {$a'$};
		\node (A'') at (0,0) {$a''$};
		\node (B) at (1.5,2) {$b+_{c}d$};
		\node (B') at (1.5,1) {$b'+_{c'}d'$};
		\node (B'') at (1.5,0) {$b''+_{c''}d''$};
		\node (C) at (3,2) {$e$};
		\node (C') at (3,1) {$e'$};
		\node (C'') at (3,0) {$e''$};
		\path[->,font=\scriptsize,>=angle 90]
		(A) edge node[above]{} (B)
		(A') edge node[above]{} (B')
		(A'') edge node[above]{} (B'')
		(C) edge node[above]{} (B)
		(C') edge node[above]{} (B')
		(C'') edge node[above]{} (B'')
		(A') edge node[left]{} (A)
		(A') edge node[left]{} (A'')
		(B') edge node[left]{} (B)
		(B') edge node[left]{} (B'')
		(C') edge node[left]{} (C)
		(C') edge node[left]{} (C'');	
		\end{tikzpicture}
	}
	\]
	This respects identities and 
	we separate the proof that 
	$\odot$ respects composition 
	into the next lemma.  
	It is straightforward to check 
	that the required equations are satisfied.  
	The associators and unitors 
	arise from universal properties.  
\end{proof}

\begin{lem}
	The assignment $\odot$ from 
	Lemma \ref{lem:SpanCospanDoubleCat} 
	preserves composition. 
	In particular, $\odot$ is a functor.
\end{lem}

\begin{proof}
	Let $\alpha, \alpha^\prime, \beta, \beta^\prime$ 
	be the following 2-morphisms
	\[
	\begin{tikzpicture}
	\node () at (-0.75,1) {$\alpha =$};
	\node (A) at (0,2) {$a$};
	\node (A') at (0,1) {$a'$};
	\node (A'') at (0,0) {$\ell$};
	\node (B) at (1,2) {$b$};
	\node (B') at (1,1) {$b'$};
	\node (B'') at (1,0) {$m$};
	\node (C) at (2,2) {$c$};
	\node (C') at (2,1) {$c'$};
	\node (C'') at (2,0) {$n$};
	\path[->,font=\scriptsize,>=angle 90]
	(A) edge node[above]{$$} (B)
	(A') edge node[above]{$$} (B')
	(A'') edge node[above]{$$} (B'')
	(C) edge node[above]{$$} (B)
	(C') edge node[above]{$$} (B')
	(C'') edge node[above]{$$} (B'')
	(A') edge node[left]{$\cong$} (A)
	(A') edge node[left]{$\cong$} (A'')
	(B') edge[->] node[left]{} (B)
	(B') edge[->] node[left]{} (B'')
	(C') edge node[left]{$\cong$} (C)
	(C') edge node[left]{$\cong$} (C'');	
	\end{tikzpicture}
	\quad \quad
	\begin{tikzpicture}
	\node () at (-0.75,1) {$\alpha' =$};
	\node (A) at (0,2) {$c$};
	\node (A') at (0,1) {$c'$};
	\node (A'') at (0,0) {$n$};
	\node (B) at (1,2) {$d$};
	\node (B') at (1,1) {$d'$};
	\node (B'') at (1,0) {$p$};
	\node (C) at (2,2) {$e$};
	\node (C') at (2,1) {$e'$};
	\node (C'') at (2,0) {$q$};
	\path[->,font=\scriptsize,>=angle 90]
	(A) edge node[above]{$$} (B)
	(A') edge node[above]{$$} (B')
	(A'') edge node[above]{$$} (B'')
	(C) edge node[above]{$$} (B)
	(C') edge node[above]{$$} (B')
	(C'') edge node[above]{$$} (B'')
	(A') edge node[left]{$\cong$} (A)
	(A') edge node[left]{$\cong$} (A'')
	(B') edge[->] node[left]{} (B)
	(B') edge[->] node[left]{} (B'')
	(C') edge node[left]{$\cong$} (C)
	(C') edge node[left]{$\cong$} (C'');	
	\end{tikzpicture}
	\]
	\[
	\begin{tikzpicture}
	\node () at (-0.75,1) {$\beta =$};
	\node (A) at (0,2) {$\ell$};
	\node (A') at (0,1) {$v'$};
	\node (A'') at (0,0) {$v$};
	\node (B) at (1,2) {$m$};
	\node (B') at (1,1) {$w'$};
	\node (B'') at (1,0) {$w$};
	\node (C) at (2,2) {$n$};
	\node (C') at (2,1) {$x'$};
	\node (C'') at (2,0) {$x$};
	\path[->,font=\scriptsize,>=angle 90]
	(A) edge node[above]{$$} (B)
	(A') edge node[above]{$$} (B')
	(A'') edge node[above]{$$} (B'')
	(C) edge node[above]{$$} (B)
	(C') edge node[above]{$$} (B')
	(C'') edge node[above]{$$} (B'')
	(A') edge node[left]{$\cong$} (A)
	(A') edge node[left]{$\cong$} (A'')
	(B') edge[->] node[left]{} (B)
	(B') edge[->] node[left]{} (B'')
	(C') edge node[left]{$\cong$} (C)
	(C') edge node[left]{$\cong$} (C'');	
	\end{tikzpicture}
	\quad \quad
	\begin{tikzpicture}
	\node () at (-0.75,1) {$\beta' =$};
	\node (A) at (0,2) {$n$};
	\node (A') at (0,1) {$x'$};
	\node (A'') at (0,0) {$x$};
	\node (B) at (1,2) {$p$};
	\node (B') at (1,1) {$y'$};
	\node (B'') at (1,0) {$y$};
	\node (C) at (2,2) {$q$};
	\node (C') at (2,1) {$z'$};
	\node (C'') at (2,0) {$z$};
	\path[->,font=\scriptsize,>=angle 90]
	(A) edge node[above]{$$} (B)
	(A') edge node[above]{$$} (B')
	(A'') edge node[above]{$$} (B'')
	(C) edge node[above]{$$} (B)
	(C') edge node[above]{$$} (B')
	(C'') edge node[above]{$$} (B'')
	(A') edge node[left]{$\cong$} (A)
	(A') edge node[left]{$\cong$} (A'')
	(B') edge[->] node[left]{} (B)
	(B') edge[->] node[left]{} (B'')
	(C') edge node[left]{$\cong$} (C)
	(C') edge node[left]{$\cong$} (C'');	
	\end{tikzpicture}
	\]
	Our goal is to show that
	\begin{equation}
	\label{eq:InterchangeCspSpan}
	(\alpha \odot \alpha') \circ (\beta \odot \beta')
	=
	(\alpha \circ \beta) \odot (\alpha' \circ \beta').
	\end{equation}
	The left hand side 
	of this equation 
	corresponds to horizontal composition 
	before vertical composition.
	The right hand side corresponds
	to composing in the
	opposite order. 
	
	First, compute the left hand side
	of \eqref{eq:InterchangeCspSpan}. 
	Composing horizontally, 
	$\alpha \odot \alpha'$ 
	and $\beta \odot \beta'$ are, respectively,
	\[
	\begin{tikzpicture}
	\node (A) at (1,1) {$a$};
	\node (A') at (1,0) {$a'$};
	\node (A'') at (1,-1) {$\ell$};
	\node (B) at (3,1) {$b+_{c}d$};
	\node (B') at (3,0) {$b'+_{c'}d'$};
	\node (B'') at (3,-1) {$m+_{n}p$};
	\node (C) at (5,1) {$e$};
	\node (C') at (5,0) {$e'$};
	\node (C'') at (5,-1) {$q$};
	\path[->,font=\scriptsize,>=angle 90]
	(A) edge node[above]{$$} (B)
	(A') edge node[above]{$$} (B')
	(A'') edge node[above]{$$} (B'')
	(C) edge node[above]{$$} (B)
	(C') edge node[above]{$$} (B')
	(C'') edge node[above]{$$} (B'')
	(A') edge node[left]{$\cong$} (A)
	(A') edge node[left]{$\cong$} (A'')
	(B') edge[->] node[left]{} (B)
	(B') edge[->] node[left]{} (B'')
	(C') edge node[left]{$\cong$} (C)
	(C') edge node[left]{$\cong$} (C'');	
	\end{tikzpicture}
	\quad \quad
	\begin{tikzpicture}
	\node (A) at (1,1) {$\ell$};
	\node (A') at (1,0) {$v'$};
	\node (A'') at (1,-1) {$v$};
	\node (B) at (3,1) {$m+_{n}p$};
	\node (B') at (3,0) {$w'+_{x'}y'$};
	\node (B'') at (3,-1) {$w+_{x}y$};
	\node (C) at (5,1) {$q$};
	\node (C') at (5,0) {$z'$};
	\node (C'') at (5,-1) {$z$};
	\path[->,font=\scriptsize,>=angle 90]
	(A) edge node[above]{$$} (B)
	(A') edge node[above]{$$} (B')
	(A'') edge node[above]{$$} (B'')
	(C) edge node[above]{$$} (B)
	(C') edge node[above]{$$} (B')
	(C'') edge node[above]{$$} (B'')
	(A') edge node[left]{$\cong$} (A)
	(A') edge node[left]{$\cong$} (A'')
	(B') edge[->] node[left]{} (B)
	(B') edge[->] node[left]{} (B'')
	(C') edge node[left]{$\cong$} (C)
	(C') edge node[left]{$\cong$} (C'');	
	\end{tikzpicture}
	\]
	The vertical composite
	$(\alpha \odot \alpha^\prime) \circ (\beta \odot \beta^\prime)$ 
	is equal to
	\begin{equation}
	\label{diag:IntrchngHorVertCspSpan}
	\raisebox{-0.5\height}{
		\begin{tikzpicture}
		\node (A) at (1,1) {$a$};
		\node (A') at (1,0) {$a'\times_{\ell}v'$};
		\node (A'') at (1,-1) {$v$};
		\node (B) at (5,1) {$b+_{d}d$};
		\node (B') at (5,0) {$(b'+_{c'}d') \times_{(m+_{n}p)} (w'+_{x'}y')$};
		\node (B'') at (5,-1) {$w+_{x}y$};
		\node (C) at (9,1) {$e$};
		\node (C') at (9,0) {$e' +_{q}z'$};
		\node (C'') at (9,-1) {$z$};
		\path[->,font=\scriptsize,>=angle 90]
		(A) edge node[above]{$$} (B)
		(A') edge node[above]{$$} (B')
		(A'') edge node[above]{$$} (B'')
		(C) edge node[above]{$$} (B)
		(C') edge node[above]{$$} (B')
		(C'') edge node[above]{$$} (B'')
		(A') edge node[left]{$\cong$} (A)
		(A') edge node[left]{$\cong$} (A'')
		(B') edge[->] node[left]{} (B)
		(B') edge[->] node[left]{} (B'')
		(C') edge node[left]{$\cong$} (C)
		(C') edge node[left]{$\cong$} (C'');	
		\end{tikzpicture}
	}
	\end{equation}
	
	Now solving for the right hand side of 
		\eqref{eq:InterchangeCspSpan},
	$\alpha \circ \beta$ and 
	$\alpha' \circ \beta'$ are
	respectively
	\[
	\begin{tikzpicture}
	\node (A) at (1,1) {$a$};
	\node (A') at (1,0) {$a' \times_{\ell}v'$};
	\node (A'') at (1,-1) {$v$};
	\node (B) at (3,1) {$b$};
	\node (B') at (3,0) {$b' \times_{m}w'$};
	\node (B'') at (3,-1) {$w$};
	\node (C) at (5,1) {$c$};
	\node (C') at (5,0) {$c' \times_{n}x'$};
	\node (C'') at (5,-1) {$x$};
	\path[->,font=\scriptsize,>=angle 90]
	(A) edge node[above]{$$} (B)
	(A') edge node[above]{$$} (B')
	(A'') edge node[above]{$$} (B'')
	(C) edge node[above]{$$} (B)
	(C') edge node[above]{$$} (B')
	(C'') edge node[above]{$$} (B'')
	(A') edge node[left]{$\cong$} (A)
	(A') edge node[left]{$\cong$} (A'')
	(B') edge[->] node[left]{} (B)
	(B') edge[->] node[left]{} (B'')
	(C') edge node[left]{$\cong$} (C)
	(C') edge node[left]{$\cong$} (C'');	
	\end{tikzpicture}
	\quad \quad 
	\begin{tikzpicture}
	\node (A) at (1,1) {$c$};
	\node (A') at (1,0) {$c' \times_{n}x'$};
	\node (A'') at (1,-1) {$x$};
	\node (B) at (3,1) {$d$};
	\node (B') at (3,0) {$d' \times_{p}y'$};
	\node (B'') at (3,-1) {$y$};
	\node (C) at (5,1) {$e$};
	\node (C') at (5,0) {$e' \times_{q}z'$};
	\node (C'') at (5,-1) {$z$};
	\path[->,font=\scriptsize,>=angle 90]
	(A) edge node[above]{$$} (B)
	(A') edge node[above]{$$} (B')
	(A'') edge node[above]{$$} (B'')
	(C) edge node[above]{$$} (B)
	(C') edge node[above]{$$} (B')
	(C'') edge node[above]{$$} (B'')
	(A') edge node[left]{$\cong$} (A)
	(A') edge node[left]{$\cong$} (A'')
	(B') edge[->] node[left]{} (B)
	(B') edge[->] node[left]{} (B'')
	(C') edge node[left]{$\cong$} (C)
	(C') edge node[left]{$\cong$} (C'');	
	\end{tikzpicture}
	\]
	The vertical composite 
	$(\alpha \circ \beta) \odot (\alpha^\prime \circ \beta^\prime)$ 
	is equal to 
	\begin{equation}
	\label{diag:IntrchngVertHorCspSpan}
	\raisebox{-0.5\height}{
		\begin{tikzpicture}
		\node (A) at (1,1) {$a$};
		\node (A') at (1,0) {$a' \times_{\ell}v'$};
		\node (A'') at (1,-1) {$v$};
		\node (B) at (5,1) {$b +_{c}d$};
		\node (B') at (5,0) {$(b'\times_{m}w')+_{(c'\times_{n}x')}(d'\times_{p}y')$};
		\node (B'') at (5,-1) {$w+_{x}y$};
		\node (C) at (9,1) {$e$};
		\node (C') at (9,0) {$e' \times_{q}z'$};
		\node (C'') at (9,-1) {$z$};
		\path[->,font=\scriptsize,>=angle 90]
		(A) edge node[above]{$$} (B)
		(A') edge node[above]{$$} (B')
		(A'') edge node[above]{$$} (B'')
		(C) edge node[above]{$$} (B)
		(C') edge node[above]{$$} (B')
		(C'') edge node[above]{$$} (B'')
		(A') edge node[left]{$\cong$} (A)
		(A') edge node[left]{$\cong$} (A'')
		(B') edge[->] node[left]{} (B)
		(B') edge[->] node[left]{} (B'')
		(C') edge node[left]{$\cong$} (C)
		(C') edge node[left]{$\cong$} (C'');	
		\end{tikzpicture}
	}
	\end{equation}
	
	Since \eqref{diag:IntrchngHorVertCspSpan} 
	and \eqref{diag:IntrchngVertHorCspSpan} 
	have coinciding outer squares, 
	they represent the same parallel class,
	hence \eqref{eq:InterchangeCspSpan} holds.
\end{proof}

Our next step is to show that 
the (braided, symmetric) monoidal structure 
from $\cat{C}$ lifts to $\dblspcsp{C}$.
We point to 
	\cite[Def.~2.9]{Shulman_ConstructSMBicats} 
for the definition of a monoidal double category.

\begin{lem}
	\label{lem:SpanCospanSM}
	The (braided, symmetric) monoidal structure 
	of $\cat{C}$ lifts to $\dblspcsp{C}$.
\end{lem}

\begin{proof}
	Again, denote 
	$\dblspcsp{C}$ by $\widetilde{\dblcat{C}}$. 
	The object $\widetilde{\dblcat{C}}_0$ and 
	arrow $\widetilde{\dblcat{C}}_1$ categories are 
	(braided, symmetric) monoidal
	by taking $\otimes$ pointwise.  
	The monoidal structure for 
	$\widetilde{\dblcat{C}}_0$-objects follows 
	from that on $\cat{C}$ and 
	for $\widetilde{\dblcat{C}}_0$-morphisms is
	\[
	(a \gets b \to c) \otimes (a' \gets b' \to c')
	=
	(a\otimes a' \gets b\otimes b' \to c\otimes c').
	\]
	Universal properties provide 
	the associator, unitors, 
	and coherence axioms. 
	It is clear that $\widetilde{\dblcat{C}}_0$ is 
	also braided or symmetric monoidal 
	whenever $\cat{C}$ is.
	
	We obtain a monoidal structure 
	for $\widetilde{\dblcat{C}}_1$-objects by 
	\[
		(a \to b \gets c) \otimes (a' \to b' \gets c')
		=
		(a\otimes a' \to b\otimes b'  \gets c\otimes c')
	\]
	and for $\widetilde{\dblcat{C}}_1$-morphisms by
	\[
	\raisebox{-0.5\height}{
		\begin{tikzpicture}
		\node (A) at (0,2) {$\bullet$};
		\node (A') at (0,1) {$\bullet$};
		\node (A'') at (0,0) {$\bullet$};
		\node (B) at (1,2) {$\bullet$};
		\node (B') at (1,1) {$\bullet$};
		\node (B'') at (1,0) {$\bullet$};
		\node (C) at (2,2) {$\bullet$};
		\node (C') at (2,1) {$\bullet$};
		\node (C'') at (2,0) {$\bullet$};
		\path[->,font=\scriptsize,>=angle 90]
		(A) edge node[above]{} (B)
		(A') edge node[above]{} (B')
		(A'') edge node[above]{} (B'')
		(C) edge node[above]{} (B)
		(C') edge node[above]{} (B')
		(C'') edge node[above]{} (B'')
		(A') edge node[left]{} (A)
		(A') edge node[left]{} (A'')
		(B') edge[->] node[left]{} (B)
		(B') edge[->] node[left]{} (B'')
		(C') edge node[left]{} (C)
		(C') edge node[left]{} (C'');	
		\end{tikzpicture}
	}
	\otimes
	\raisebox{-0.5\height}{
		\begin{tikzpicture}
		\node (A) at (0,2) {$\ast$};
		\node (A') at (0,1) {$\ast$};
		\node (A'') at (0,0) {$\ast$};
		\node (B) at (1,2) {$\ast$};
		\node (B') at (1,1) {$\ast$};
		\node (B'') at (1,0) {$\ast$};
		\node (C) at (2,2) {$\ast$};
		\node (C') at (2,1) {$\ast$};
		\node (C'') at (2,0) {$\ast$};
		\path[->,font=\scriptsize,>=angle 90]
		(A) edge node[above]{} (B)
		(A') edge node[above]{} (B')
		(A'') edge node[above]{} (B'')
		(C) edge node[above]{} (B)
		(C') edge node[above]{} (B')
		(C'') edge node[above]{} (B'')
		(A') edge node[left]{} (A)
		(A') edge node[left]{} (A'')
		(B') edge[->] node[left]{} (B)
		(B') edge[->] node[left]{} (B'')
		(C') edge node[left]{} (C)
		(C') edge node[left]{} (C'');	
		\end{tikzpicture}
	}
	=
	\raisebox{-0.5\height}{
		\begin{tikzpicture}
		\node (A) at (0,2) {$\bullet\otimes \ast$};
		\node (A') at (0,1) {$\bullet\otimes \ast$};
		\node (A'') at (0,0) {$\bullet\otimes \ast$};
		\node (B) at (1.5,2) {$\bullet\otimes \ast$};
		\node (B') at (1.5,1) {$\bullet\otimes \ast$};
		\node (B'') at (1.5,0) {$\bullet\otimes \ast$};
		\node (C) at (3,2) {$\bullet\otimes \ast$};
		\node (C') at (3,1) {$\bullet\otimes \ast$};
		\node (C'') at (3,0) {$\bullet\otimes \ast$};
		\path[->,font=\scriptsize,>=angle 90]
		(A) edge node[above]{} (B)
		(A') edge node[above]{} (B')
		(A'') edge node[above]{} (B'')
		(C) edge node[above]{} (B)
		(C') edge node[above]{} (B')
		(C'') edge node[above]{} (B'')
		(A') edge node[left]{} (A)
		(A') edge node[left]{} (A'')
		(B') edge[>->] node[left]{} (B)
		(B') edge[>->] node[left]{} (B'')
		(C') edge node[left]{} (C)
		(C') edge node[left]{} (C'');	
		\end{tikzpicture}
	}
	\]
	The monoidal unit for $\widetilde{\dblcat{C}}_1$ 
	is the identity cospan on $I$ 
	which is exactly $U_I$. 
	Universal properties again provide 
	the associator, unitors, and coherence axioms. 
	The braiding and symmetry of the 
	monoidal structure clearly lifts from $\cat{C}$.
	
	It is straightforward to check
	that the source and target functors 
	are strict monoidal and respect the 
	associator and unitors. 
	It remains to find two 
	invertible globular 2-cells:
	one witnessing interchange 
	\[
		r \from 
			(M_1 \otimes N_1) \odot (M_2 \otimes N_2)
			\to 
			(M_1 \odot M_2) \otimes (N_1 \odot N_2)
	\]
	for $\widetilde{\dblcat{C}}_1$-objects $M_i$ and $N_i$, 
	and another witnessing units
	\[
		u \from U_{f \otimes g} \to U_f \otimes U_g 
	\]
	for $\widetilde{\dblcat{C}}_0$-arrows $f$ and $g$. 
	Moreover, $r$ and $u$ must satisfy 
	certain axioms 
	\cite[Def.~2.9]{Shulman_ConstructSMBicats}.	
	
	If $M_1 = (a \to b \gets c)$, 
	$M_2 = (c \to d \gets e)$, 
	$N_1 = (v \to w \gets x)$, and 
	$N_2 = (x \to y \gets z)$, 
	then $r$ has domain 
	\[
		a \otimes v 
		\to 
		(b \otimes w)+_{(c \otimes x)} (d \otimes y)
		\gets
		(e \otimes z)
	\]
	and codomain
	\[
		a \otimes v 
		\to 
		b+_c d \otimes w+_x y
		\gets
		(e \otimes z).
	\]
	We must find a 2-cell 
	whose outer square
	is formed by the 
	domain and codomain of $r$
	on the top and bottom
	plus identity $\widetilde{\dblcat{C}}_0$-morphisms
	on the left and right.  
	Let $J \from \cat{D} \to \cat{C} \times \cat{C}$ 
	be the functor on the category
	\[
		\cat{D} = \{ \bullet \to \bullet \gets \bullet \to \bullet \gets \bullet\}.
	\] 
	whose image is of the form
	\[
		(a \to b \gets c \to d \gets e ) \times (w \to v \gets x \to y \gets z ).
	\]
	Then the domain of $r$ is 
	$\colim (\Delta \circ \otimes)$ 
	and the codomain is 
	$\otimes \left( \colim (\Delta) \right)$. 
	But these are isomorphic by assumption. 
	This gives $r$.
	
	To define $u$, 
	let $f$ be the 
	$\widetilde{\dblcat{C}}_0$-morphism 
	$a \gets b \to c$ and 
	let $g$ be $x \gets y \to z$.  
	It is easy to check that 
	both $U_{f \otimes g}$ and $U_{f} \otimes U_g$ are
	\[
	\begin{tikzpicture}
		\node (A) at (0,2) {$a\otimes x$};
		\node (A') at (0,1) {$a\otimes x$};
		\node (A'') at (0,0) {$a\otimes x$};
		\node (B) at (2,2) {$b\otimes y$};
		\node (B') at (2,1) {$b\otimes y$};
		\node (B'') at (2,0) {$b\otimes y$};
		\node (C) at (4,2) {$c\otimes z$};
		\node (C') at (4,1) {$c\otimes z$};
		\node (C'') at (4,0) {$c\otimes z$};
		\path[->,font=\scriptsize,>=angle 90]
		(A) edge node[above]{} (B)
		(A') edge node[above]{} (B')
		(A'') edge node[above]{} (B'')
		(C) edge node[above]{} (B)
		(C') edge node[above]{} (B')
		(C'') edge node[above]{} (B'')
		(A') edge node[left]{} (A)
		(A') edge node[left]{} (A'')
		(B') edge node[left]{} (B)
		(B') edge node[left]{} (B'')
		(C') edge node[left]{} (C)
		(C') edge node[left]{} (C'');
	\end{tikzpicture}
	\]
	where the legs of the horizontal cospans 
	are built using the inverses 
	of the legs of $f$ and $g$. 
	The legs of the vertical spans are identities.
	
	As for the remaining axioms, 
	they are straightforward though tedious 
	to check and are left to the reader.
\end{proof}

\begin{lem}
	\label{lem:SpanCospanIsofibrant}
	$\dblspcsp{C}$ is isofibrant.  
\end{lem}

\begin{proof}
	Take a vertical morphism 
	$f = (a \gets b \to c)$. 
	The legs of the companion 
	$\widehat{f} = (a \to b \gets c)$, 
	are the inverses of those
	from $f$. 
	The companion is equipped with the 2-morphisms
	\[
	\raisebox{-0.5\height}{
		\begin{tikzpicture}
		\node (A) at (0,2) {$a$};
		\node (A') at (0,1) {$b$};
		\node (A'') at (0,0) {$c$};
		\node (B) at (1,2) {$b$};
		\node (B') at (1,1) {$c$};
		\node (B'') at (1,0) {$c$};
		\node (C) at (2,2) {$c$};
		\node (C') at (2,1) {$c$};
		\node (C'') at (2,0) {$c$};
		\path[->,font=\scriptsize,>=angle 90]
		(A) edge node[above]{} (B)
		(A') edge node[above]{} (B')
		(A'') edge node[above]{} (B'')
		(C) edge node[above]{} (B)
		(C') edge node[above]{} (B')
		(C'') edge node[above]{} (B'')
		(A') edge node[left]{} (A)
		(A') edge node[left]{} (A'')
		(B') edge node[left]{} (B)
		(B') edge node[left]{} (B'')
		(C') edge node[left]{} (C)
		(C') edge node[left]{} (C'');
		\end{tikzpicture}
	}
	\t{ and }
	\raisebox{-0.5\height}{
		\begin{tikzpicture}
		\node (A) at (0,2) {$a$};
		\node (A') at (0,1) {$a$};
		\node (A'') at (0,0) {$a$};
		\node (B) at (1,2) {$a$};
		\node (B') at (1,1) {$a$};
		\node (B'') at (1,0) {$b$};
		\node (C) at (2,2) {$a$};
		\node (C') at (2,1) {$b$};
		\node (C'') at (2,0) {$c$};
		\path[->,font=\scriptsize,>=angle 90]
		(A) edge node[above]{} (B)
		(A') edge node[above]{} (B')
		(A'') edge node[above]{} (B'')
		(C) edge node[above]{} (B)
		(C') edge node[above]{} (B')
		(C'') edge node[above]{} (B'')
		(A') edge node[left]{} (A)
		(A') edge node[left]{} (A'')
		(B') edge node[left]{} (B)
		(B') edge node[left]{} (B'')
		(C') edge node[left]{} (C)
		(C') edge node[left]{} (C'');
		\end{tikzpicture}
	}
	\]
	The reader may check that 
	the required equations hold.
	The conjoint $\check{f} $ of 
	$f$ is $\widehat{f}^{\text{op}}$. 
\end{proof}

\begin{thm}
	\label{thm:SpansCospasAreSMBicat}
	$\spcsp{C}$ is a (braided, symmetric) monoidal bicategory.
\end{thm}

\begin{proof}
	Apply Theorem \ref{thm:DoubleGivesBi}.
\end{proof}

This proves the first half of Theorem 
	\ref{thm:SpCspC_is_SMCC_bicategory}. 
To prove the second half, 
we assume 
that $\cat{C}$ is a 
cocartesian monoidal bicategory.

Note that we take Stay's 
definition of a 
compact closed bicategory
	\cite{Stay_CompactClosedBicats}. 

\begin{lem}
	\label{lem:PushoutDiagram}
	The diagram
	\[
	\begin{tikzpicture}
	\node (UL) at (0,1) {$x+x+x$};
	\node (LL) at (0,0) {$x+x$};
	\node (UR) at (3,1) {$x+x$};
	\node (LR) at (3,0) {$x$};
	\path[->,font=\scriptsize,>=angle 90]
	(UL) edge node[above] {$x+\nabla$} (UR)
	(UL) edge node[left] {$\nabla +x$} (LL)
	(UR) edge node[right] {$\nabla$} (LR)
	(LL) edge node[above] {$\nabla$} (LR);
	\end{tikzpicture}
	\]
	in $\cat{C}$ is a pushout square.
\end{lem}

\begin{proof}
	Suppose
	$f,g \from x+x \to y$ 
	form a cocone over the above diagram. 
	Let $\iota \from x \to x+x+x$ be an inclusion
	into the middle copy of $x$. 
	Observe that 
		$\ell \coloneqq (\nabla + x) \circ \iota$ and 
		$r \coloneqq (x + \nabla) \circ \iota$ 
	are	the left and right inclusions $x \to x+x$. 
	Then $f \circ \ell = g \circ r$ is a map $x \to y$, 
	which we claim to be the unique map 
	making the required diagram commute. 
	Indeed, given $h \from x \to y$ such that 
	$f = h \circ \nabla = g$, then 
	$g \circ r = f \circ \ell = h \circ \nabla \circ \ell = h$.
\end{proof}

\begin{thm}
	\label{thm:SpansCospansAreCCBicat}
	$\spcsp{C}$ is compact closed.
\end{thm}

\begin{proof}
	The objects are self dual.
	To show this, 
	start with an object $x$.  
	Define the evaluation morphism and 
	coevaluation morphism by
	\[
		e = (x+x \xto{\nabla} x \gets 0), \quad \quad 
		c = (0 \to x \xleftarrow{\nabla} x+x).
	\]
	We next define the 
	cusp isomorphisms, 
	$\alpha$ and $\beta$.
	The domain for $\alpha$ 
	is the composite
	\[
		x \xto{\ell}
		x+x \xleftarrow{x+\nabla}
		x+x+x \xto{\nabla +x}
		x+x \xleftarrow{r}
		x
	\]
	and the codomain for $\beta$ is
	\[
		x \xto{r}
		x+x \xleftarrow{\nabla+x}
		x+x+x \xto{x+\nabla}
		x+x \xleftarrow{\ell}
		x.
	\]
	That these are both identity cospans on $x$
	follows from Lemma \ref{lem:PushoutDiagram}
	and the equations $\nabla+x = \ell \circ \nabla$ 
	and $x + \nabla = r \circ \nabla$ 
	Take $\alpha$ and $\beta$ each to be 
	the identity 2-morphism on $x$. 
	This gives a dual pair 
	$(x,x,e,c,\alpha,\beta)$ 
	which we can complete 
	to a coherent dual pair 
	\cite[p.~22]{Pstrski_DualObjectsCobord}. 
\end{proof}

\nocite{*}
\bibliographystyle{eptcs}
\bibliography{3--Catfying_zxCalc--Bibliography}

\begin{thebibliography}{10}
\providecommand{\bibitemdeclare}[2]{}
\providecommand{\surnamestart}{}
\providecommand{\surnameend}{}
\providecommand{\urlprefix}{Available at }
\providecommand{\url}[1]{\texttt{#1}}
\providecommand{\href}[2]{\texttt{#2}}
\providecommand{\urlalt}[2]{\href{#1}{#2}}
\providecommand{\doi}[1]{doi:\urlalt{http://dx.doi.org/#1}{#1}}
\providecommand{\bibinfo}[2]{#2}

\bibitemdeclare{inproceedings}{AbramCoecke_CatSemanticQuantum}
\bibitem{AbramCoecke_CatSemanticQuantum}
\bibinfo{author}{S.~\surnamestart Abramsky\surnameend} \&
  \bibinfo{author}{B.~\surnamestart Coecke\surnameend} (\bibinfo{year}{2004}):
  \emph{\bibinfo{title}{A categorical semantics of quantum protocols}}.
\newblock In: {\sl \bibinfo{booktitle}{Logic in Computer Science. Proceedings
  of the 19th Annual IEEE Symposium}}, pp. \bibinfo{pages}{415--425},
  \doi{10.1109/LICS.2004.1319636}.

\bibitemdeclare{article}{CoeckeEdwards_ToyQuantumCategories}
\bibitem{CoeckeEdwards_ToyQuantumCategories}
\bibinfo{author}{Bob \surnamestart B.~Coecke\surnameend} \&
  \bibinfo{author}{B.~\surnamestart Edwards\surnameend} (\bibinfo{year}{2011}):
  \emph{\bibinfo{title}{Toy quantum categories}}.
\newblock {\sl \bibinfo{journal}{Electron. Notes Theor. Comput. Sci.}}
  \bibinfo{volume}{270}(\bibinfo{number}{1}),
  \doi{10.1016/j.entcs.2011.01.004}.
\newblock \urlprefix\url{https://arxiv.org/abs/0808.1037}.

\bibitemdeclare{misc}{Backens_Completeness}
\bibitem{Backens_Completeness}
\bibinfo{author}{M.~\surnamestart Backens\surnameend} (\bibinfo{year}{2016}):
  \emph{\bibinfo{title}{Completeness and the ZX-calculus}},
  \doi{10.1109/LICS.2004.1319636}.
\newblock \urlprefix\url{https://arxiv.org/abs/1602.08954}.

\bibitemdeclare{misc}{BaezCoyaFong_Props}
\bibitem{BaezCoyaFong_Props}
\bibinfo{author}{J.~\surnamestart Baez\surnameend},
  \bibinfo{author}{B.~\surnamestart Coya\surnameend} \&
  \bibinfo{author}{F.~\surnamestart Rebro\surnameend} (\bibinfo{year}{2017}):
  \emph{\bibinfo{title}{Props in network theory}}.
\newblock \urlprefix\url{https://arxiv.org/abs/1707.08321}.

\bibitemdeclare{misc}{BaezFong_CompPassLinNets}
\bibitem{BaezFong_CompPassLinNets}
\bibinfo{author}{J.~\surnamestart Baez\surnameend} \&
  \bibinfo{author}{B.~\surnamestart Fong\surnameend} (\bibinfo{year}{2015}):
  \emph{\bibinfo{title}{A compositional framework for passive linear
  networks}}.
\newblock \urlprefix\url{https://arxiv.org/abs/1504.05625}.

\bibitemdeclare{article}{BaezFongPollard_CompMarkovProcesses}
\bibitem{BaezFongPollard_CompMarkovProcesses}
\bibinfo{author}{J.~\surnamestart Baez\surnameend},
  \bibinfo{author}{B.~\surnamestart Fong\surnameend} \&
  \bibinfo{author}{B.~\surnamestart Pollard\surnameend} (\bibinfo{year}{2016}):
  \emph{\bibinfo{title}{A compositional framework for {M}arkov processes}}.
\newblock {\sl \bibinfo{journal}{J. Math. Phys.}} \bibinfo{volume}{57},
  \doi{10.1063/1.4941578}.

\bibitemdeclare{misc}{BarKissingerVicary_Globular}
\bibitem{BarKissingerVicary_Globular}
\bibinfo{author}{K.~\surnamestart Bar\surnameend},
  \bibinfo{author}{A.~\surnamestart Kissinger\surnameend} \&
  \bibinfo{author}{J.~\surnamestart Vicary\surnameend} (\bibinfo{year}{2016}):
  \emph{\bibinfo{title}{Globular: an online proof assistant for
  higher-dimensional rewriting}}.
\newblock \urlprefix\url{http://globular.science}.

\bibitemdeclare{misc}{Cicala_SpansCospans}
\bibitem{Cicala_SpansCospans}
\bibinfo{author}{D.~\surnamestart Cicala\surnameend} (\bibinfo{year}{2016}):
  \emph{\bibinfo{title}{Spans of cospans}}.
\newblock \urlprefix\url{https://arxiv.org/abs/1611.07886}.

\bibitemdeclare{misc}{CicalaCourser_BicatSpansCospan}
\bibitem{CicalaCourser_BicatSpansCospan}
\bibinfo{author}{D.~\surnamestart Cicala\surnameend} \&
  \bibinfo{author}{K.~\surnamestart Courser\surnameend} (\bibinfo{year}{2017}):
  \emph{\bibinfo{title}{Spans of cospans in a topos}}.
\newblock \urlprefix\url{https://arxiv.org/abs/1707.02098}.

\bibitemdeclare{incollection}{CoeckeDuncan_QuantumObsInitialReport}
\bibitem{CoeckeDuncan_QuantumObsInitialReport}
\bibinfo{author}{B.~\surnamestart Coecke\surnameend} \&
  \bibinfo{author}{R.~\surnamestart Duncan\surnameend} (\bibinfo{year}{2008}):
  \emph{\bibinfo{title}{Interacting quantum observables}}.
\newblock In: {\sl \bibinfo{booktitle}{Automata, languages and programming.
  {P}art {II}}}, {\sl \bibinfo{series}{Lecture Notes in Comput. Sci.}}
  \bibinfo{volume}{5126}, \bibinfo{publisher}{Springer, Berlin}, pp.
  \bibinfo{pages}{298--310}, \doi{10.1007/978-3-540-70583-3\_25}.

\bibitemdeclare{article}{CoeckeDuncan_QuantumObsFullPaper}
\bibitem{CoeckeDuncan_QuantumObsFullPaper}
\bibinfo{author}{B.~\surnamestart Coecke\surnameend} \&
  \bibinfo{author}{R.~\surnamestart Duncan\surnameend} (\bibinfo{year}{2011}):
  \emph{\bibinfo{title}{Interacting quantum observables: categorical algebra
  and diagrammatics}}.
\newblock {\sl \bibinfo{journal}{New J. Phys.}} \bibinfo{volume}{13},
  \doi{10.1088/1367-2630/13/4/043016}.

\bibitemdeclare{incollection}{CoeckeEdwards_ToyTheories}
\bibitem{CoeckeEdwards_ToyTheories}
\bibinfo{author}{B.~\surnamestart Coecke\surnameend} \&
  \bibinfo{author}{B.~\surnamestart Edwards\surnameend} (\bibinfo{year}{2012}):
  \emph{\bibinfo{title}{Spekkens's toy theory as a category of processes}}.
\newblock In: {\sl \bibinfo{booktitle}{Mathematical foundations of information
  flow}}, {\sl \bibinfo{series}{Proc. Sympos. Appl.
  Math.}}~\bibinfo{volume}{71}, \bibinfo{publisher}{Amer. Math. Soc.}, pp.
  \bibinfo{pages}{61--68}, \doi{10.1090/psapm/071/602}.

\bibitemdeclare{article}{CoeckeEdwardsSpekkens_PhaseGrpsNonLocality}
\bibitem{CoeckeEdwardsSpekkens_PhaseGrpsNonLocality}
\bibinfo{author}{B.~\surnamestart Coecke\surnameend},
  \bibinfo{author}{B.~\surnamestart Edwards\surnameend} \&
  \bibinfo{author}{R.~\surnamestart Spekkens\surnameend}
  (\bibinfo{year}{2011}): \emph{\bibinfo{title}{Phase groups and the origin of
  non-locality for qubits}}.
\newblock {\sl \bibinfo{journal}{Electron. Notes Theor. Comput. Sci.}}
  \bibinfo{volume}{270}(\bibinfo{number}{2}),
  \doi{10.1016/j.entcs.2011.01.021}.

\bibitemdeclare{incollection}{CoeckePavlovic_QuanMeasSums}
\bibitem{CoeckePavlovic_QuanMeasSums}
\bibinfo{author}{B.~\surnamestart Coecke\surnameend} \&
  \bibinfo{author}{D.~\surnamestart Pavlovic\surnameend}
  (\bibinfo{year}{2008}): \emph{\bibinfo{title}{Quantum measurements without
  sums}}.
\newblock In: {\sl \bibinfo{booktitle}{Mathematics of quantum computation and
  quantum technology}}, \bibinfo{series}{Chapman \& Hall CRC Appl. Math.
  Nonlinear Sci. Ser.}, \bibinfo{publisher}{Chapman \& Hall/CRC, Boca Raton,
  FL}, pp. \bibinfo{pages}{559--596}, \doi{10.1201/9781584889007.ch16}.

\bibitemdeclare{article}{CoeckePavVicary_OrthogBases}
\bibitem{CoeckePavVicary_OrthogBases}
\bibinfo{author}{B.~\surnamestart Coecke\surnameend},
  \bibinfo{author}{D.~\surnamestart Pavlovic\surnameend} \&
  \bibinfo{author}{J.~\surnamestart Vicary\surnameend} (\bibinfo{year}{2013}):
  \emph{\bibinfo{title}{A new description of orthogonal bases}}.
\newblock {\sl \bibinfo{journal}{Math. Structures Comput. Sci.}}
  \bibinfo{volume}{23}(\bibinfo{number}{3}), \doi{10.1017/S0960129512000047}.

\bibitemdeclare{article}{CoeckePerdix_EnvironClassicChannels}
\bibitem{CoeckePerdix_EnvironClassicChannels}
\bibinfo{author}{B.~\surnamestart Coecke\surnameend} \&
  \bibinfo{author}{S.~\surnamestart Perdrix\surnameend} (\bibinfo{year}{2012}):
  \emph{\bibinfo{title}{Environment and classical channels in categorical
  quantum mechanics}}.
\newblock {\sl \bibinfo{journal}{Log. Methods Comput. Sci.}}
  \bibinfo{volume}{8}(\bibinfo{number}{4}), \doi{10.2168/LMCS-8(4:14)2012}.

\bibitemdeclare{incollection}{Corradini_AlgApprGraphTrans}
\bibitem{Corradini_AlgApprGraphTrans}
\bibinfo{author}{A.~\surnamestart Corradini\surnameend},
  \bibinfo{author}{U.~\surnamestart Montanari\surnameend},
  \bibinfo{author}{F.~\surnamestart Rossi\surnameend},
  \bibinfo{author}{H.~\surnamestart Ehrig\surnameend},
  \bibinfo{author}{R.~\surnamestart Heckel\surnameend} \&
  \bibinfo{author}{M.~\surnamestart Lowe\surnameend} (\bibinfo{year}{1997}):
  \emph{\bibinfo{title}{Algebraic approaches to graph transformation. {B}asic
  concepts and double pushout approach}}.
\newblock In: {\sl \bibinfo{booktitle}{Handbook of graph grammars and computing
  by graph transformation, {V}ol. 1}}, \bibinfo{publisher}{World Sci. Publ.,
  River Edge, NJ}, \doi{10.1142/9789812384720\_0003}.

\bibitemdeclare{article}{DanosKashefiPanang_MeasurementCalc}
\bibitem{DanosKashefiPanang_MeasurementCalc}
\bibinfo{author}{V.~\surnamestart Danos\surnameend},
  \bibinfo{author}{E.~\surnamestart Kashefi\surnameend} \&
  \bibinfo{author}{P.~\surnamestart Panangaden\surnameend}
  (\bibinfo{year}{2007}): \emph{\bibinfo{title}{The measurement calculus}}.
\newblock {\sl \bibinfo{journal}{J. ACM}}
  \bibinfo{volume}{54}(\bibinfo{number}{2}), \doi{10.1145/1219092.1219096}.

\bibitemdeclare{misc}{DixonDuncanKissinger_QuantomaticWebsite}
\bibitem{DixonDuncanKissinger_QuantomaticWebsite}
\bibinfo{author}{L.~\surnamestart Dixon\surnameend},
  \bibinfo{author}{R.~\surnamestart Duncan\surnameend} \&
  \bibinfo{author}{A.~\surnamestart Kissinger\surnameend}:
  \emph{\bibinfo{title}{Quantomatic}}.
\newblock \urlprefix\url{https://sites.google.com/site/quantomatic/}.

\bibitemdeclare{incollection}{DuncanPerdix_GraphStatesEulerDecomp}
\bibitem{DuncanPerdix_GraphStatesEulerDecomp}
\bibinfo{author}{R.~\surnamestart Duncan\surnameend} \&
  \bibinfo{author}{S.~\surnamestart Perdrix\surnameend} (\bibinfo{year}{2009}):
  \emph{\bibinfo{title}{Graph states and the necessity of {E}uler
  decomposition}}.
\newblock In: {\sl \bibinfo{booktitle}{Mathematical theory and computational
  practice}}, {\sl \bibinfo{series}{Lecture Notes in Comput. Sci.}}
  \bibinfo{volume}{5635}, \bibinfo{publisher}{Springer, Berlin}, pp.
  \bibinfo{pages}{167--177}, \doi{10.1007/978-3-642-03073-4\_18}.

\bibitemdeclare{article}{DuncanPerdrix_RewritingQuantumCompu}
\bibitem{DuncanPerdrix_RewritingQuantumCompu}
\bibinfo{author}{R.~\surnamestart Duncan\surnameend} \&
  \bibinfo{author}{S.~\surnamestart Perdrix\surnameend} (\bibinfo{year}{2010}):
  \emph{\bibinfo{title}{Rewriting measurement-based quantum computations with
  generalised flow}}.
\newblock {\sl \bibinfo{journal}{Automata, Languages and Programming}},
  \doi{10.1007/978-3-642-14162-1\_24}.

\bibitemdeclare{misc}{EvansDuncanLangPanan_ClassMutualUnbias}
\bibitem{EvansDuncanLangPanan_ClassMutualUnbias}
\bibinfo{author}{J.~\surnamestart Evans\surnameend},
  \bibinfo{author}{R.~\surnamestart Duncan\surnameend},
  \bibinfo{author}{A.~\surnamestart Lang\surnameend} \&
  \bibinfo{author}{P.~\surnamestart Panangaden\surnameend}
  (\bibinfo{year}{2009}): \emph{\bibinfo{title}{Classifying all mutually
  unbiased bases in Rel}}.
\newblock \urlprefix\url{https://arxiv.org/abs/0909.4453}.

\bibitemdeclare{misc}{Fong_AlgOpenSystems}
\bibitem{Fong_AlgOpenSystems}
\bibinfo{author}{B.~\surnamestart Fong\surnameend} (\bibinfo{year}{2016}):
  \emph{\bibinfo{title}{The Algebra of Open and Interconnected Systems}}.
\newblock \urlprefix\url{https://arxiv.org/abs/arXiv:1609.05382}.

\bibitemdeclare{article}{Habel}
\bibitem{Habel}
\bibinfo{author}{A.~\surnamestart Habel\surnameend},
  \bibinfo{author}{J.~\surnamestart Muller\surnameend} \&
  \bibinfo{author}{D.~\surnamestart Plump\surnameend} (\bibinfo{year}{2001}):
  \emph{\bibinfo{title}{Double-pushout graph transformation revisited}}.
\newblock {\sl \bibinfo{journal}{Math. Structures Comput. Sci.}}
  \bibinfo{volume}{11}(\bibinfo{number}{5}), \doi{10.1017/S0960129501003425}.

\bibitemdeclare{article}{BaezPollard_CompFrameRxNets}
\bibitem{BaezPollard_CompFrameRxNets}
\bibinfo{author}{John \surnamestart J.~Baez\surnameend} \&
  \bibinfo{author}{B.~\surnamestart Pollard\surnameend} (\bibinfo{year}{2017}):
  \emph{\bibinfo{title}{A compositional framework for reaction networks}}.
\newblock {\sl \bibinfo{journal}{Rev. Math. Phys.}}
  \bibinfo{volume}{29}(\bibinfo{number}{9}), \doi{10.1142/S0129055X17500283}.

\bibitemdeclare{article}{JoyalStreet_GeomTensorCalc}
\bibitem{JoyalStreet_GeomTensorCalc}
\bibinfo{author}{A.~\surnamestart Joyal\surnameend} \&
  \bibinfo{author}{R.~\surnamestart Street\surnameend} (\bibinfo{year}{1991}):
  \emph{\bibinfo{title}{The geometry of tensor calculus. {I}}}.
\newblock {\sl \bibinfo{journal}{Adv. Math.}}
  \bibinfo{volume}{88}(\bibinfo{number}{1}),
  \doi{10.1016/0001-8708(91)90003-P}.

\bibitemdeclare{article}{Kissinger_Pictures}
\bibitem{Kissinger_Pictures}
\bibinfo{author}{A.~\surnamestart Kissinger\surnameend} (\bibinfo{year}{2012}):
  \emph{\bibinfo{title}{Pictures of processes: automated graph rewriting for
  monoidal categories and applications to quantum computing}}.
\newblock {\sl \bibinfo{journal}{Ph.D. Thesis{,} University of Oxford}}.
\newblock \urlprefix\url{https://arxiv.org/abs/1203.0202}.

\bibitemdeclare{incollection}{KissingerZamd_Quantomatic}
\bibitem{KissingerZamd_Quantomatic}
\bibinfo{author}{A.~\surnamestart Kissinger\surnameend} \&
  \bibinfo{author}{V.~\surnamestart Zamdzhiev\surnameend}
  (\bibinfo{year}{2015}): \emph{\bibinfo{title}{Quantomatic: a proof assistant
  for diagrammatic reasoning}}.
\newblock In: {\sl \bibinfo{booktitle}{Automated deduction---{CADE} 25}}, {\sl
  \bibinfo{series}{Lecture Notes in Comput. Sci.}} \bibinfo{volume}{9195},
  \bibinfo{publisher}{Springer, Cham}, pp. \bibinfo{pages}{326--336},
  \doi{10.1007/978-3-319-21401-6\_22}.

\bibitemdeclare{article}{Dixon_OpenGraphs}
\bibitem{Dixon_OpenGraphs}
\bibinfo{author}{Lucas \surnamestart L.~Dixon\surnameend},
  \bibinfo{author}{R.~\surnamestart Duncan\surnameend} \&
  \bibinfo{author}{A.~\surnamestart Kissinger\surnameend}
  (\bibinfo{year}{2010}): {\sl \bibinfo{journal}{Electron. Proc. Theor. Comput.
  Sci.}} \bibinfo{volume}{26}, \doi{10.4204/EPTCS.26.16}.

\bibitemdeclare{book}{MaclaneMoerdijk}
\bibitem{MaclaneMoerdijk}
\bibinfo{author}{S.~\surnamestart MacLane\surnameend} \&
  \bibinfo{author}{I.~\surnamestart Moerdijk\surnameend}
  (\bibinfo{year}{2012}): \emph{\bibinfo{title}{Sheaves in geometry and logic:
  A first introduction to topos theory}}.
\newblock \bibinfo{publisher}{Springer Science \& Business Media}.

\bibitemdeclare{article}{Merry_BangGraphs}
\bibitem{Merry_BangGraphs}
\bibinfo{author}{A.~\surnamestart Merry\surnameend} (\bibinfo{year}{2014}):
  \emph{\bibinfo{title}{Reasoning with !-Graphs}}.
\newblock {\sl \bibinfo{journal}{CoRR}} \bibinfo{volume}{abs/1403.7828}.
\newblock \urlprefix\url{http://arxiv.org/abs/1403.7828}.

\bibitemdeclare{book}{NielsenChuang_QuantumCompInfo}
\bibitem{NielsenChuang_QuantumCompInfo}
\bibinfo{author}{M.~\surnamestart Nielsen\surnameend} \&
  \bibinfo{author}{I.~\surnamestart Chuang\surnameend} (\bibinfo{year}{2010}):
  \emph{\bibinfo{title}{Quantum computation and quantum information}}.
\newblock \bibinfo{publisher}{Cambridge University Press, Cambridge},
  \doi{10.1017/CBO9780511976667}.

\bibitemdeclare{incollection}{Pavlovic_QuanClassNondetermCompu}
\bibitem{Pavlovic_QuanClassNondetermCompu}
\bibinfo{author}{D.~\surnamestart Pavlovic\surnameend} (\bibinfo{year}{2009}):
  \emph{\bibinfo{title}{Quantum and classical structures in nondeterministic
  computation}}.
\newblock In: {\sl \bibinfo{booktitle}{Quantum interaction}}, {\sl
  \bibinfo{series}{Lecture Notes in Comput. Sci.}} \bibinfo{volume}{5494},
  \bibinfo{publisher}{Springer, Berlin}, pp. \bibinfo{pages}{143--157},
  \doi{10.1007/978-3-642-00834-4\_13}.

\bibitemdeclare{incollection}{Penrose_NegDimTensors}
\bibitem{Penrose_NegDimTensors}
\bibinfo{author}{R.~\surnamestart Penrose\surnameend} (\bibinfo{year}{1971}):
  \emph{\bibinfo{title}{Applications of negative dimensional tensors}}.
\newblock In: {\sl \bibinfo{booktitle}{Combinatorial {M}athematics and its
  {A}pplications ({P}roc. {C}onf., {O}xford, 1969)}},
  \bibinfo{publisher}{Academic Press, London}, pp. \bibinfo{pages}{221--244}.
\newblock \urlprefix\url{http://homepages.math.uic.edu/~kauffman/Penrose.pdf}.

\bibitemdeclare{article}{Pollard_OpenMarkov}
\bibitem{Pollard_OpenMarkov}
\bibinfo{author}{B.~\surnamestart Pollard\surnameend} (\bibinfo{year}{2016}):
  \emph{\bibinfo{title}{Open Markov Processes: A Compositional Perspective on
  Non{-}Equilibrium Steady States in Biology}}.
\newblock {\sl \bibinfo{journal}{Entropy}}
  \bibinfo{volume}{18}(\bibinfo{number}{4}), \doi{10.3390/e18040140}.

\bibitemdeclare{misc}{Pstrski_DualObjectsCobord}
\bibitem{Pstrski_DualObjectsCobord}
\bibinfo{author}{P.~\surnamestart Pstragowski\surnameend}
  (\bibinfo{year}{2014}): \emph{\bibinfo{title}{On dualizable objects in
  monoidal bicategories, framed surfaces and the {C}obordism {H}ypothesis}}.
\newblock \urlprefix\url{https://arxiv.org/abs/1411.6691}.

\bibitemdeclare{article}{SassoneSobocinski_PetriNets}
\bibitem{SassoneSobocinski_PetriNets}
\bibinfo{author}{V.~\surnamestart Sassone\surnameend} \&
  \bibinfo{author}{P.~\surnamestart Sobocinski\surnameend}
  (\bibinfo{year}{2005}): \emph{\bibinfo{title}{A congruence for Petri nets}}.
\newblock {\sl \bibinfo{journal}{Electronic Notes in Theoretical Computer
  Science}} \bibinfo{volume}{127}(\bibinfo{number}{2}),
  \doi{10.1016/j.entcs.2005.02.008}.

\bibitemdeclare{incollection}{Selinger_GraphicsMonCats}
\bibitem{Selinger_GraphicsMonCats}
\bibinfo{author}{P.~\surnamestart Selinger\surnameend} (\bibinfo{year}{2011}):
  \emph{\bibinfo{title}{A survey of graphical languages for monoidal
  categories}}.
\newblock In: {\sl \bibinfo{booktitle}{New structures for physics}}, {\sl
  \bibinfo{series}{Lecture Notes in Phys.}} \bibinfo{volume}{813},
  \bibinfo{publisher}{Springer, Heidelberg}, pp. \bibinfo{pages}{289--355},
  \doi{10.1007/978-3-642-12821-9\_4}.

\bibitemdeclare{misc}{Shulman_ConstructSMBicats}
\bibitem{Shulman_ConstructSMBicats}
\bibinfo{author}{M.~\surnamestart Shulman\surnameend} (\bibinfo{year}{2010}):
  \emph{\bibinfo{title}{Constructing symmetric monoidal bicategories}}.
\newblock \urlprefix\url{http://arxiv.org/abs/1004.0993}.

\bibitemdeclare{article}{Stay_CompactClosedBicats}
\bibitem{Stay_CompactClosedBicats}
\bibinfo{author}{M.~\surnamestart Stay\surnameend} (\bibinfo{year}{2016}):
  \emph{\bibinfo{title}{Compact closed bicategories}}.
\newblock {\sl \bibinfo{journal}{Theory Appl. Categ.}} \bibinfo{volume}{31}.
\newblock \urlprefix\url{https://arxiv.org/abs/1301.1053}.

\end{thebibliography}
 
\end{document}